\documentclass[twoside,12pt]{article}
\usepackage{amsmath}
\usepackage[dvipsnames,usenames]{color}
\usepackage{amsfonts}
\usepackage{mathrsfs}
\usepackage{amssymb}
\usepackage{amsthm}
\def\<{\langle}
\def\>{\rangle}
\numberwithin{equation}{section}
\def\<{\langle}
\def\>{\rangle}

\def\-{\overline}

\def\-{\overline}

\usepackage{fancyhdr}
\newtheorem{theorem}{Theorem}[section]
\newtheorem{lemma}[theorem]{Lemma}
\newtheorem{Corollary}[theorem]{Corollary}

\newtheorem{Definition}[theorem]{Definition}

\newtheorem{example}[theorem]{Example}
\newtheorem{remark}[theorem]{Remark}

\begin{document}
\title{\bf On the  $C^{\infty}$ version of the reflection principle for mappings between CR manifolds}

\author{S. Berhanu \thanks{Work supported in part by NSF DMS 1300026}\\Department of Mathematics\\Temple University\\Philadelphia, PA 19122 \and Ming Xiao\\Department of Mathematics\\Rutgers University\\New Brunswick, NJ 08854}

  \maketitle

\begin{abstract}
We prove a smooth version of the classical Schwarz reflection principle for CR \newline mappings between an abstract CR manifold $M$ and a generic CR manifold embedded in euclidean complex space. As a consequence of our results, we settle a conjecture of \newline X. Huang formulated in [Hu2].
\end{abstract}


\section{Introduction}
In this paper we study the regularity problem for CR mappings between
CR manifolds where the CR dimension of the source manifold is less than or equal to that of the target manifold. Our results imply a positive answer to a conjecture of X. Huang in [Hu2] and provide a solution to a question raised in [Fr1] (see Corollaries 2.10 and 2.11).

One of our theorems can be viewed as a smooth version of the analyticity
theorems of Forstneric ([Fr1]) and Huang [Hu1-2] for CR mappings between CR manifolds of differing dimensions. The article is devoted to results along the line of research on establishing the  smooth version of the  Schwarz reflection
principle for holomorphic maps in several variables.
Results of this type were first proved in the 70's starting with the
work
of Fefferman [Fe], Lewy [Le] and Pinchuk [Pi]. The seminal work [BJT] has influenced a lot of work on the subject. For extensive surveys
and many references on this research, the reader may consult
 the articles by Bedford [Be], Forstneric [Fr2], and Bell-Narasimhan [BN]. Among the many related papers we mention here [CKS], [CGS], [CS], [DW], [EH], [EL], [Fr1], [Fr3], [Hu1], [Hu2], [KP], [K], [La1], [La2], [La3], [M], [NWY], [Tu], and [W]. In [Fr3] Forstneric generalized Fefferman's theorem to CR homeomorphisms $f:M\to M'$ where $f^{-1}$ is CR,  $M$ and $M'$ are generic CR submanifolds of $\Bbb C^n$ with the same CR dimension. The book [BER] by
Baouendi, Ebenfelt, and Rothschild contains a detailed account and
references related to the study when the manifolds are real analytic
or real algebraic.

We prove results on the smoothness of CR maps where the source manifold $M$ is assumed to be an abstract (not necessarily embeddable) CR manifold. We mention that the results are new even when $M$ is embeddable. Our first main result, Theorem $2.3$, generalizes to abstract CR manifolds a theorem of Lamel in [La1] proved for generic CR manifolds embedded in complex spaces. The second main result, Theorem $2.5$, establishes the smoothness on a dense open subset of a $C^k$ CR mapping $F: (M,\mathcal V)\rightarrow (M',\mathcal V')$ where $(M,\mathcal V)$ is an abstract CR manifold of CR dimension $n$ and $M'\subset \Bbb C^{n+k}$ is a hypersurface that is strongly pseudoconvex. A condition on the Levi form of $(M,\mathcal V)$ is assumed in Theorem $2.5$.

Our approach is  based on the framework established by Roberts
[GR] in his thesis and a later paper by Lamel in [La1]. The notion of $k_{0}-$nondegeneracy of a CR mapping (Definition $2.1$) and the ``almost holomorphic" implicit function theorem of Lamel in [La1] and [La2] play crucial roles in the proofs and formulations of our results. The proof is
also motivated by the study of the real  analyticity for CR maps
between real analytic strongly pseudoconvex hypersurfaces of
different dimensions in Forstneric [Fr1] and Huang [Hu1]. We mention
that in [Fr1], Forstneric conjectured that $F$ must be real analytic
when $M_1\subset \Bbb C^{n+1}$ and $M_2\subset \Bbb C^{n+k}$ ($k\geq 2, n \geq 1$) are real analytic hypersurfaces with $M_1$ of finite type, $M_2$ strongly pseudoconvex, and he proved that this is indeed the case on a dense open set when $F$ is smooth. The conjecture of Forstneric was settled  by Huang  ([Hu1]) who obtained the analyticity of $F$ on a dense open subset assuming only that $F\in C^{k}$. The analyticity of $F$, when both $M_1$ and $M_2$ are as in [Fr1] and when $F$ is  only $C^{k}$-smooth also follows from Theorem $2.5$ in this paper and Forstneric's analyticity result when $F$ is smooth.

The paper is organized as follows. In Section $2$ we state the main results and prove a preliminary microlocal regularity result that is used in the proof of Theorem $2.5$ and supplies a good class of examples to which Theorem $2.3$ can be applied. Section $3$ contains the proof of Theorem $2.3$ and Theorem $2.5$ is proved in Section $4$. In an appendix we indicate why we focus only on CR mappings where the target manifold has a higher CR dimension than that of the source manifold.

{\bf{Acknowledgement:}} We use this opportunity to thank the anonymous referee for his helpful comments that have resulted in an improvement of the presentation of the results in this paper. The second author expresses his gratefulness to his advisor Xiaojun Huang for his constant help and encouragement. He also thanks Xiaoshan Li for his interest and help in
this work.
\section{Statements of the results and the proof of a \\ preliminary result}

Let $M$ be a smooth manifold and let $\mathcal{V}$ be a subbundle of $\mathbb{C}TM,$ the complexified tangent bundle of $M.$ The pair $(M,\mathcal{V})$ is called an abstract CR manifold if $\mathcal{V}$ is involutive and if for each $p \in M,$
$\mathcal{V}_{p} \cap \overline{\mathcal{V}}_{p}=0.$ Recall that the involutivity of $\mathcal{V}$ means that the space of smooth sections of $\mathcal{V},C^{\infty}(M,\mathcal{V}),$ is closed under commutators. Let $n$ be the complex dimension of the fibers $\mathcal{V}_{p}$ and write $\mathrm{dim}_{\mathbb{R}}M=2n+d.$ The number $n$ is called the CR dimension of $M,$ and  $d$ is called the CR codimension of $M.$ If $d=1,$ the CR manifold is said to be of hypersurface type. A smooth section of $\mathcal{V}$ is called a CR vector field and a function (or distribution)
is called CR if $Lf=0$ for any CR vector field $L.$ The CR manifold $(M,\mathcal{V})$ is called locally embeddable
if for any $p_{0} \in M,$  there exist $m$ complex-valued $C^{\infty}$ functions $Z_{1},\cdots,Z_{m}$ defined near $p_{0}$ with $m=\mathrm{dim}_{\mathbb{R}}M-n,\,$ such that
the $Z_{j}$ are CR functions near $p_{0},$ and the differentials $dZ_{1},\cdots,dZ_{m}$ are $\mathbb{C}-$linearly independent. In this case, the mapping
$$p \mapsto Z(p)=(Z_{1}(p),\cdots,Z_{m}(p)) \in \mathbb{C}^{m}=\mathbb{C}^{n+d}$$
is an immersion near $p_{0}.$ Thus, if $U$ is a small neighborhood of $p_{0},$ then $Z(U)$ is an embedded submanifold of $\mathbb{C}^{m}$ and is a generic CR submanifold of $\mathbb{C}^{m}$ whose induced CR bundle agrees with the push forward $Z_{*}(\mathcal{V})$ (see [BER] and [J] for more details).
Let $(M',\mathcal{V}')$ be another abstract CR manifold with CR dimension $n'$ and CR codimension $d'.$ A CR mapping of class $C^{k}(k \geq 1)~H:(M,\mathcal{V}) \rightarrow (M',\mathcal{V}')$ is a $C^{k}$ mapping $H: M\rightarrow M'$ such that for each $p \in M,$
$$dH(\mathcal{V}_{p}) \subset \mathcal{V}'_{H(p)}.$$

When $(M',\mathcal{V}')$ is a generic CR submanifold of $\mathbb{C}^{N'}(N'=n'+d'),$ then a $C^{k}$ mapping $H=(H_{1},\cdots,H_{N'}):M \rightarrow M'$ is a CR mapping if and only if each $H_{j}$ is a CR function.
One of our main results generalizes to an abstract CR manifold $(M,\mathcal{V})$ a regularity theorem of Lamel ([La1])
for CR mappings of embedded CR manifolds. We need to recall from [La1] the notion of nondegenerate CR mappings. Let $\widetilde{M} \subset \mathbb{C}^{N}$ and $\widetilde{M}' \subset \mathbb{C}^{N'}$ be two generic CR submanifolds of $\mathbb{C}^{N}$ and $\mathbb{C}^{N'}$ respectively. If $d$ and $d'$ denote the real codimensions of $\widetilde{M}$ and $\widetilde{M}',$ then $n=N-d$ and $n'=N'-d'$ are the CR dimensions of $\widetilde{M}$ and $\widetilde{M}'$ respectively. Let $H: \widetilde{M} \rightarrow \widetilde{M}'$ be a CR mapping of class $C^{k}.$
\begin{Definition} ([La1])
Let $\widetilde{M}, \widetilde{M}'$ and $H$ be as above and $p_{0} \in \widetilde{M}.$ Let $\rho=(\rho_{1},\cdots,\rho_{d'})$ be local defining functions for $\widetilde{M}'$ near $H(p_{0}),$ and choose a basis $L_{1},\cdots,L_{n}$ of CR vector fields for $\widetilde{M}$ near $p_{0}.$ If $\alpha=(\alpha_{1},\cdots,\alpha_{n})$ is a multiindex, write $L^{\alpha}=L_{1}^{\alpha_{1}} \cdots L_{n}^{\alpha_{n}}.$ Define the increasing sequence of subspaces $E_{l}(p_{0}) (0 \leq l \leq k)$ of $\mathbb{C}^{N'}$ by

$$E_{l}(p_{0})=\mathrm{Span}_{\mathbb{C}}\{L^{\alpha} \rho_{\mu,Z'}(H(Z),\overline{H(Z)})|_{Z=p_{0}}: 0 \leq |\alpha| \leq l, 1 \leq \mu \leq d'\}.$$
Here $\rho_{\mu,Z'}=(\frac{\partial \rho_{\mu}}{\partial z'_{1}},\cdots,\frac{\partial \rho_{\mu}}{\partial z'_{N'}}),$ and $Z'=(z'_{1},\cdots,z'_{N'})$ are the coordinates in $\mathbb{C}^{N'}.$ We say that $H$ is $k_{0}-$nondegenerate at $p_{0}\,\,(0 \leq k_{0} \leq k)$ if
$$E_{k_{0}-1}(p_{0}) \neq E_{k_{0}}(p_{0}) = \mathbb{C}^{N'}.$$
\end{Definition}
The dimension of $E_l(p)$ over $\Bbb C$ will be called the $l^{\text{th}}$ geometric rank of $F$ at $p$ and it will be denoted by $\text{rank}_l(F,p)$.

For the invariance of this definition under the choice of the defining functions $\rho_{\mu},$ the basis of CR vector fields and the choice of holomorphic coordinates in $\mathbb{C}^{N'},$ the reader is referred to [La2]. An intrinsic definition was presented in the paper [EL]. If $M$ is a manifold for which the identity map is $k_0-$nondegenerate, then the manifold is called  {\it $k_0-$nondegenerate}. This latter notion was introduced for embedded hypersurfaces in [BHR] and it is shown in [E] that it can be formulated for an abstract CR manifold. The reader is referred to these two references for a detailed treatment of this concept and its connection with holomorphic nondegeneracy in the sense of Stanton ([S]). In particular, in [BHR] and [E] it is shown that Levi-nondegeneracy of a CR manifold is equivalent to $1-$nondegeneracy. Thus the notion of $k_0-$nondegeneracy of a CR manifold can be viewed as a generalization of Levi nondegeneracy.\newline\newline The main result in [La1] is as follows:

\begin{theorem}\label{thml1}
Let $M \subset \mathbb{C}^{N}, M' \subset \mathbb{C}^{N'}$ be smooth generic submanifolds of $\mathbb{C}^{N}$ and
$\mathbb{C}^{N'}$ respectively, $p_{0} \in M, H=(H_{1},\cdots,H_{N'}): M \rightarrow M'$ a $C^{k_{0}}$ CR map which is $k_{0}-$nondegenerate at $p_{0}$ and extends continuously to a holomorphic map in a wedge $W$ with edge $M.$ Then $H$ is smooth in some neighborhood of $p_{0}.$
\end{theorem}

Here recall that if $p_{0} \in M$, $d=$ the CR codimension of $M$, and $U \subset \mathbb{C}^{N}$ is a neighborhood of $p_{0},$ a wedge $W$ with edge $M$ centered at $p_{0}$ is defined to be an open set of the form:
$$W=\{Z \in U: r(Z,\overline{Z}) \in \Gamma \},$$ where  $\Gamma \subset \mathbb{R}^{d}$ is an open convex cone, and
 $r=(r_{1},\cdots,r_{d})$ are defining functions for $M$ near $p_{0}.$
In section 3, we will prove the following generalization of Theorem \ref{thml1}.

\begin{theorem}\label{thm12}
Let $(M,\mathcal{V})$ be an abstract CR manifold and $M' \subset \mathbb{C}^{N'}$ a generic CR submanifold of $\mathbb{C}^{N'}.$ Let $H=(H_{1},\cdots,H_{N'}): M \rightarrow M'$ be a CR mapping of class $C^{k_{0}}$ which is
$k_{0}-$nondegenerate at $p_{0}$ and assume that for some open convex cone $\Gamma \subset \mathbb{R}^{d},$
$$\mathrm{WF}(H_{j})|_{p_{0}} \subset \Gamma, j=1,\cdots,N'$$
where $d$ is the CR codimension of $M.$ Then $H$ is $C^{\infty}$ in some neighborhood of $p_{0}.$
\end{theorem}

In Theorem \ref{thm12},  $\mathrm{WF}(u)$ denotes the $C^{\infty}$ wave front set of $u,$ that is,
$$\mathrm{WF}(u)=\{\sigma\in T^{\ast}M: u~\text{is not microlocally smooth at}~\sigma\}. $$
For details about the $C^{\infty}$ wave front set of a function, see [H].
\begin{remark}
In Theorem \ref{thml1}, the assumption that $H$ is the boundary value of a holomorphic function in a wedge  implies the much weaker condition that $\text{WF}(H_{j})|_{p_{0}} \subset \Gamma$ for some $\Gamma$ as in Theorem \ref{thm12}. Indeed, in the embedded case as in Theorem \ref{thml1}, a CR function $h$ on $M$ is the boundary value of a holomorphic function in a wedge if and only if its hypo-analytic wave front set is contained in an acute cone which means that the FBI transform of $h$ decays exponentially. Our assumption in Theorem \ref{thm12} only requires the FBI transform to decay rapidly.
\end{remark}
In what follows, given a CR manifold $(M, \mathcal{V})$, $T^0$ will denote its characteristic bundle, that is, $T^0=\{\sigma\in T^{\ast}M:\langle \sigma, L \rangle =0\,\,\text{for every smooth section $L$ of}\,\, \mathcal{V}\}$. 
\begin{theorem}\label{thm}
Let $(M,\mathcal{V})$ be an abstract CR manifold with CR dimension $n \geq 1$ such that the Levi form at every covector $\sigma\in T^0$ has a nonzero eigenvalue. Suppose $M' \subset \mathbb{C}^{n+k}$ is a hypersurface  that is strongly pseudoconvex ($k\geq 1$) and let $\mathcal V'$ denote the CR bundle of $M'$. Let $F=(F_{1},\cdots,F_{n+k}): M \rightarrow M'$ be a CR mapping of class $C^{k}$ whose differential $dF:\mathcal V_p\rightarrow \mathcal V_{F(p)}' $ is injective at every $p\in M$. Then $F$ is $C^{\infty}$ on a dense open subset of $M$.
\end{theorem}

We note that the preceding theorem allows a weakening of the smoothness assumption in Theorem 1.2 of [EL] on finite jet determination. The theorem also implies that some of the results in [BR] hold under a weaker smoothness assumption on the CR maps involved. If $M\subset \Bbb C^N, M'\subset \Bbb C^{N'}$ are hypersurfaces, with $M$ Levi nondegenerate at $p\in M$ and $F:M\rightarrow M'$ is a CR mapping which is transversal at $p$, that is, $dF(\Bbb CT_pM)$ is not contained in $\mathcal V'_{F(p)}+\overline{\mathcal V'_{F(p)}}$, then $F$ is a local embedding (see section 3.4 in [EL]). Many other situations where $(M,\mathcal{V})$ and $(M',\mathcal V')$ are as in Theorem \ref{thm} and $dF$ is injective can be found in the work [BR].

Let $M, M', F$ be as in Theorem 2.5. Define
$$\Omega_{1}=\{p \in M: \mathrm{rank}_{k}(F,p)=n+k \},$$
$$\Omega_{2}=\{p \in M: \mathrm{rank}_{k}(F,q) \leq n+k-1~\text{at all points}~q~ \text{in some neighborhood of}~p~\text{in}~M \},$$
$$\Omega=\{ p \in M: F~\text{is smooth in a neighborhood of}~p~\text{in}~M\}.$$
For each $1 \leq l \leq k,$ we also set
$$S_{l}:=\{p \in M:\mathrm{rank}_{l}(F,p) \leq n+l-1 \}.$$
Note that $\Omega_{1},\Omega_{2},\Omega$ are all open in $M,$ and $\Omega_{1} \cup \Omega_{2}$ is dense in $M.$ Moreover, $\Omega_{2} \subset \overset{\circ}{S}_{k}.$

\begin{Definition}
Let $M,~M',~F,~S_{l}$ be as above. For any $p \in \Omega_{2},$ we define the degenerate degree of $F$ at $p$ to be
$$\min \{1 \leq l \leq k: \text{there exists a neighborhood}~O~\text{of}~p~\text{such that}~O \subset S_{l} \},$$
and write it as $\mathrm{deg}(F,p).$
\end{Definition}

\begin{remark}
 Definition 2.6 is independent of the choices of the defining function, the basis of CR vector fields and the choice of holomorphic coordinates in $\mathbb{C}^{n+k}.$ For any $p \in \Omega_{2},$ by Lemma $4.1$ in Section $4$, $\mathrm{rank}_{1}(F,p)=n+1,$  which yields that $\mathrm{deg}(F,p) \geq  2.$ Moreover, by the definition of $\mathrm{deg}(F,p),$ if we let $d=\mathrm{deg}(F,p),$ there exists a neighborhood $\widetilde{O}$ of $p$ in $M$ and $\{p_{i}\}_{0}^{\infty} \subset \widetilde{O}$~such that $\mathrm{rank}_{d}(F,q) \leq n+d-1$ for all $q \in \widetilde{O},~\{p_{i}\}$ converges to p, and $\mathrm{rank}_{d-1}(F,p_{i})=n+d-1$ for all $i \geq 0.$
\end{remark}

\bigskip

Theorem \ref{thm} will follow from the following Theorem and Theorem \ref{thmn29} below which together with Theorem \ref{thm12}  imply that
$F$ is smooth in $\Omega_{1},$ that is,  $\Omega_{1} \subset \Omega.$

\begin{theorem}\label{thmb}
For any $p \in \Omega_{2},$ there exists a sequence
$\{p_{i}\}_{i=0}^{\infty} \subset \Omega$ that  converges to $p.$
\end{theorem}

 It follows that $\Omega$ is dense in $\Omega_{1} \cup \Omega_{2},$
  and hence in $M.$ We remark that Theorem 4.8 will show that if for some integer $l$,  $1\leq l\leq k-1$,  $\mathrm{rank}_{l+1}(F,q)=n+l$ for all points $q$ in a neighborhood of $p$,  and $\mathrm{rank}_l(F,p)=n+l$, then $F:M\rightarrow M'$ is smooth in a neighborhood of $p$, where $F,M,$ and $M'$ are as in Theorem 2.5. That is, such points $p$ are in $\Omega$.

\bigskip

Before we present the proofs of Theorem $2.3$ and Theorem $2.5$, we will prove the following result which supplies a class of examples to which Theorem \ref{thm12} applies. This theorem will also be used in the proof of Theorem 2.5. The result may be viewed as the smooth version of Hans Lewy's extendability theorem in the embedded case.

\begin{theorem}\label{thmn29}
Let $(M,\mathcal{V})$ be an abstract CR manifold, $\sigma \in T_{p}^{0},$ with the property that the Levi form at
$\sigma$ has a negative eigenvalue. Then if $u$ is a CR function (or distribution) near $p, \sigma \not \in \mathrm{WF}(u). $
In particular, if the Levi form at every covector $\eta \in T_{p}^{0}$ has a nonzero eigenvalue, then there is an open convex cone~$\Gamma \subset \mathbb{R}^{d} (d=~\text{the CR codimension of}~ M)$ such that for every CR function $u$ near $p,$ $\mathrm{WF}(u)|_{p} \subset \Gamma.$
\end{theorem}
Theorem $2.5$ implies the following corollary which settles Huang's conjecture in [Hu2]:
\begin{Corollary}
Let $M\subset \Bbb C^{n+1}, \,M'\subset \Bbb C^{n+k}$ be smooth strongly pseudoconvex real hypersurfaces with $n\geq 1,k\geq 1$. Let $F:M\rightarrow M'$ be a CR mapping of class $C^k$. Then $F\in C^{\infty}(\Omega)$ on a dense open subset $\Omega\subset M$.
\end{Corollary}
Theorem $2.5$ also provides a solution to a question of Forstneric in [Fr1] using methods different from the ones employed by Huang in the solution that he gave in [Hu1]:
\begin{Corollary}
Let $M\subset \Bbb C^{N}, \,M'\subset \Bbb C^{N'}$ be real analytic hypersurfaces ($1<N<N'$), $M$ of finite type (in D'Angelo's sense) and $M'$ strongly pseudoconvex. If $F:M\rightarrow M'$ is a CR mapping of class $C^{N'-N+1}$, then $F$ extends to a holomorphic map on a neighborhood of an open, dense subset of $M$.
\end{Corollary}
\begin{proof} Let $p\in M$. If every neighborhood of $p$ contains a point where the Levi form has a positive and a negative eigenvalue, then $p$ is in the closure of the set where $F$ is smooth. We may therefore assume that  a neighborhood $D$ of $p$ is pseudoconvex. Note next that since $M$ doesn't contain a complex variety of positive dimension, it can not be Levi flat in any neighborhood of $p$. We can therefore assume that $p$ is in the closure of the set of strictly pseudoconvex points in $M$. This latter assertion can be seen by using the arguments in Lemma 6.2 in [BHR]. In that paper, $M$ was assumed algebraic but the reasoning in the Lemma is valid for $M$ as in this corollary.  Since we may assume that $F$ is non constant, at a point of strict pseudo convexity, the differential $dF$ is injective. The corollary now follows from Theorem $2.5$ and the analyticity theorem in [Fr1].
\end{proof}
In [Hu1] $M$ was assumed strongly pseudoconvex. However, as indicated above, when $M$ is of finite type in D'Angelo's sense, the problem is reduced to the strongly pseudoconvex case.

\bigskip

We now present the proof of Theorem $2.9$.
\begin{proof}
Recall that the Levi form of $(M, \mathcal{V})$ at the characteristic covector $\sigma \in T_{p}^{0}$ is the hermitian form on $\mathcal{V}$ defined by
$$\mathcal{L}_{\sigma}(v,w)=\frac{1}{2\sqrt{-1}}\langle\sigma,[L,\overline{L}']_{p}\rangle,$$
where $L$ and $L'$ are smooth sections of $\mathcal{V}$ defined near $p$ with $L(p)=v,L'(p)=w.$ When this form has a negative eigenvalue, there is a CR vector $L$ near $p$ such that
$$\frac{1}{2\sqrt{-1}}\langle\sigma,[L,\overline{L}]_{p}\rangle <0.$$

We may therefore assume that we are in coordinates $(x,t) \in \mathbb{R}^{n_{0}} \times \mathbb{R}$ that vanish at $p$,
$$L=\frac{\partial}{\partial t} + \sqrt{-1}\sum_{j=1}^{n_{0}} b_{j}(x,t)\frac{\partial}{\partial x_{j}},$$
where the $b_{j}$ are $C^{\infty}$ and real-valued functions near $(0,0), \sigma=(0,0,\xi^{0},0)$ satisfies $b(0,0) \cdot \xi^{0}=0,(b=(b_{1},\cdots,b_{n_{0}}))$ and
\begin{equation}\label{eqn11}
\left \langle(\xi^{0},0),\frac{[L,\overline{L}]_{0}}{2\sqrt{-1}}\right\rangle=-\frac{\partial b}{\partial t}(0) \cdot \xi^{0} <0.
\end{equation}
Assume that $Lu=0$ near $(0,0).$ We wish to show that $\sigma \not \in \mathrm{WF}(u).$

We introduce an additional variable $s \in \mathbb{R}$ and define
$$L_{1}=\frac{\partial}{\partial s} +\sqrt{-1} L= \frac{\partial}{\partial s}+ \sqrt{-1}\frac{\partial}{\partial t}-\sum_{j=1}^{n_{0}}b_{j}(x,t)\frac{\partial}{\partial x_{j}}.$$
Let $Z_{i}(x,t,s)\,(1 \leq i \leq n_{0})$ be $C^{\infty}$ functions near the origin satisfying
$$L_{1}Z_{i}(x,t,s)=O(s^{l}),~\text{as}~s\rightarrow 0,~\forall~l \geq 1, l \in \mathbb{N},~\text{and}~Z_{i}(x,t,0)=x_{i}.$$
Set $Z_{n_{0}+1}(x,t,s)=t-\sqrt{-1}s.$ For $1 \leq i \leq n_{0},$ we can write $Z_{i}(x,t,s)=x_{i}+s\psi_{i}(x,t,s)$ for some $C^{\infty}$ functions $\psi_{i}.$ We have, for any $l \geq 1, 1 \leq i \leq n_{0},$
\begin{equation} \label{eqn12}
s \frac{\partial \psi_{i}}{\partial s}(x,t,s)+\psi_{i}(x,t,s)+\sqrt{-1}s\frac{\partial \psi_{i}}{\partial t}(x,t,s)-\sum_{j=1}^{n_{0}}b_{j}(x,t)\left(\delta_{ij}+s\frac{\psi_{i}}{\partial x_{j}}(x,t,s)\right)=O(s^{l}).
\end{equation}
It follows that

\begin{equation}\label{eqn13}
\psi_{i}(x,t,0)=b_{i}(x,t),\,\, 1 \leq i \leq n_{0}.
\end{equation}
Differentiating equation (\ref{eqn12}) with respect to $s$ leads to,
$$s \frac{\partial^{2} \psi_{i}}{\partial s^{2}} +2 \frac{\psi_{i}}{\partial s} + \sqrt{-1}\frac{\psi_{i}}{\partial t} + \sqrt{-1} s \frac{\partial ^{2} \psi_{i}}{\partial s \partial t}- \sum_{j=1}^{n_{0}} b_{j}\frac{\partial \psi_{i}}{\partial x_{j}}- s \sum_{j=1}^{n_{0}} b_{j} \frac{\partial^{2} \psi_{i}}{\partial s \partial x_{j}}=O(s^{l}), ~\forall~l \geq 1.$$
Evaluating the latter at $s=0,$ we get, for any $1 \leq i \leq n_{0},$
$$2 \frac{\partial \psi_{i}}{\partial s}(x,t,0) + \sqrt{-1} \frac{\partial \psi_{i}}{\partial t}(x,t,0)- \sum_{j=1}^{n_{0}}b_{j}(x,t)\frac{\partial \psi_{i}}{\partial x_{j}}(x,t,0)=0,$$
which together with equation $(2.3)$ leads to:
\begin{equation}\label{eqn14}
\mathrm{Im}\,\psi_{i}(x,t,0)=0 ~\text{and}~ \frac{\partial}{\partial s}\mathrm{Im}\, \psi_{i}(x,t,0)=-\frac{1}{2}\frac{\partial b_{i}}{\partial t}(x,t),~\forall~ 1 \leq i \leq n_{0}.
\end{equation}

\bigskip

We will use the FBI transform in $(x,t)$ space.
For the solution $u=u(x,t),$ at level $s=s',$ we write,
$$\mathcal{F}(x,t,\xi,\tau,s')=\int_{\mathbb{R}^{n_{0}+1}} e^{Q(x,t,x',t',\xi,\tau,s')} \eta(x',t')u(x',t')dZ_{1}(x',t',s') \wedge \cdots \wedge dZ_{n_0+1}(x',t',s'),$$
where $(\xi,\tau) \in \mathbb{R}^{n_{0}} \times \mathbb{R}, \eta \in C_{0}^{\infty}(\mathbb{R}^{n_{0}+1}), \eta(x,t) \equiv 1$ for $|x|^2 + t^2 \leq r^2, \eta(x,t) \equiv0$ when $|x|^2 +t^2 \geq 2r^2$ for some $r>0$ to be fixed. Here
$$
Q(x,t,x',t',\xi,\tau,s')=\sqrt{-1}\langle(\xi,\tau),(x-Z(x',t',s'),t-Z_{n_{0}+1}(x',t',s'))\rangle
$$
$$
~~~~~~~~~~~~~~~~~~~~~~~~~~~~~-K|(\xi,\tau)|((x-Z(x',t',s'))^2
+(t-Z_{n_{0}+1}(x',t',s'
))^2),
$$
where $Z=(Z_{1},\cdots,Z_{n_{0}}),(x-Z(x',t',s'))^2=\sum_{j=1}^{n_{0}}(x_{j}-Z_{j}(x',t',s'))^2,$ and $K$
is a positive number which will be determined.

Let $M_{i}=\sum_{j=1}^{n_{0}} a_{ij}(x,t,s) \frac{\partial}{\partial x_{j}}, 1 \leq i \leq n_{0}$ and
$M_{n_{0}+1}=\frac{\partial}{\partial t} + \sum_{j=1}^{n_{0}} c_{j}(x,t,s) \frac{\partial}{\partial x_{j}}$ be $C^{\infty}$ vector fields near the origin in $(x,t,s)$ space that satisfy
$$M_{i}Z_{j}=\delta_{ij}, \,\,1 \leq i ,j \leq n_{0}+1.$$
For any $C^1$ function $h=h(x,t,s),$
\begin{equation}\label{eqn15}
dh=\sum_{i=1}^{n_{0}+1}M_{i}(h)dZ_{i}+\left(L_{1}h-\sum_{j=1}^{n_{0}+1}M_{j}(h)L_{1}(Z_{j})\right)ds
\end{equation}
which can be verified by applying both sides of the equation to the basis of vector fields $\{L_{1},M_{1},\cdots,M_{n_{0}+1}\}$ of $\mathbb{C}T(\mathbb{R}^{n_{0}+2}).$ Equation (\ref{eqn15}) implies that
\begin{equation}\label{eqn16}
d(hdZ_{1} \wedge \cdots \wedge dZ_{n_{0}+1})=\left(L_{1}h-\sum_{j=1}^{n_{0}+1}M_{j}(h)L_{1}(Z_{j})\right)ds \wedge dZ_{1} \wedge \cdots \wedge dZ_{n_{0}+1}.
\end{equation}
Let $q(x,t,x',t',\xi,\tau,s')=\eta(x',t')u(x',t')e^{Q(x,t,x',t',\xi,\tau,s')}.$
Denoting $dZ_{1}\wedge \cdots \wedge dZ_{n_{0}+1}$ by $dZ$ and using equation (\ref{eqn16}), we have,
\begin{equation}\label{eqn17}
d(qdZ)=\left(L_{1}(\eta u)+\eta u L_{1}(Q)-\sum_{j=1}^{n_{0}+1}(M_{j}(\eta u)+\eta u M_{j}(Q))L_{1}Z_{j}\right)e^Q ds \wedge dZ.
\end{equation}
By Stokes theorem, for $|s_{0}|$ small, we have,
\begin{equation}\label{eqn18}
\int_{\mathbb{R}^{n_{0}+1}} q(x,t,x',t',\xi,\tau,0)dx'dt'=\int_{\mathbb{R}^{n_{0}+1}}q (x,t,x',t',\xi,\tau,s_{0}) dZ(x',t',s_{0})+ \int_{0}^{s_{0}} \int_{\mathbb{R}^{n_{0}+1}} d(q dZ)
\end{equation}

We will estimate the two integrals on the right in equation $(2.8)$ for $(x,t)$ near $(0,0)$ in $\mathbb{R}^{n_{0}+1}$ and $(\xi,\tau)$ in a conic neighborhood $\Gamma$ of $(\xi^{0},0)$ in $\mathbb{R}^{n_{0}+1}.$ Observe that if $\psi=(\psi_{1},\cdots,\psi_{n_{0}}),$
\begin{equation}\label{eqn19}
\begin{split}
\mathrm{Re}~Q(x,t,x',t',\xi,\tau,s') =& s'\langle \xi,\mathrm{Im}\, \psi(x',t',s')\rangle-\tau s'\\
&-K|(\xi,\tau)|(|x-x'-s'\mathrm{Re} \,\psi(x',t',s')|^2 +|t-t'|^2-s'^2)
\end{split}
\end{equation}
Using equation (\ref{eqn14}), we can write
\begin{equation}
\begin{split}
\mathrm{Im} \,\psi(x,t,s)=&- \frac{1}{2}\frac{\partial b}{\partial t}(x,t)s +O(s^2)\\
=&-\frac{1}{2}\frac{\partial b}{\partial t}(0,0)s + O(|xs|+|ts|+s^2)
\end{split}
\end{equation}
and so plugging this into equation (\ref{eqn19}) yields
\begin{equation}\label{eqn110}
\begin{split}
\mathrm{Re}\,Q(x,t,x',t',\xi,\tau,s')  = & -\frac{1}{2}\langle\xi,\frac{\partial b}{\partial t}(0,0)\rangle s'^2-\tau s' \\
& -K|(\xi,\tau)|(|x-x'-s'\mathrm{Re}\,\psi(x',t',s')|^2 +|t-t'|^2-s'^2) \\
& +|\xi|O(|x'|s'^2+|t'|s'^2+ |s'|^3)
\end{split}
\end{equation}
Since $\langle \xi^0,\frac{\partial b}{\partial t}(0,0)\rangle >0,$ given $0 < \delta <1,$ we can get $M>0$ and a conic neighborhood $\Gamma$ of $(\xi^0,0)$ in $\mathbb{R}^{n_{0}+1}$ such that
\begin{equation}\label{eqn111}
\langle \xi,\frac{\partial b}{\partial t}(0,0) \rangle \geq M|\xi|~~\text{and}~~|\tau| < \delta |\xi|,~\text{when}~(\xi,\tau) \in \Gamma.
\end{equation}

Our interest is in estimating the integral on the left hand side of equation (\ref{eqn18}) for $(x,t)$ near $(0,0)$
and $(\xi,\tau) \in \Gamma.$ When $\tau >0,$ we take $s_{0} >0$ in (\ref{eqn18}) while when $\tau <0,$ we use $s_{0} <0.$ This together with (\ref{eqn111}) allows us to deduce the following inequality from (\ref{eqn110}):
\begin{equation}\label{eqn112}
\begin{split}
\mathrm{Re}\,Q(x,t,x',t',\xi,\tau, s')  \leq & -\frac{M}{2}s'^2|\xi|-K|(\xi,\tau)|(|x-x'-s'\mathrm{Re}\,\psi(x',t',s')|^2\\  & +|t-t'|^2-s'^2)+|\xi|O(|x'|s'^2+|t'|s'^2+ |s'|^3) \\
\leq & (-\frac{M}{2}+(1+\delta)K)s'^2|\xi|-K|\xi|(|x-x'-s'\mathrm{Re}\,\psi(x',t',s')|^2 \\
& +|t-t'|^2)+ |\xi|O(|x'|s'^2+|t'|s'^2+ |s'|^3)
\end{split}
\end{equation}
Choose $K=\frac{M}{4(1+\delta)}.$ Then (\ref{eqn112}) becomes
\begin{equation} \label{eqn113}
\begin{split}
\mathrm{Re}\,Q(x,t,x',t',\xi,\tau, s') \leq & -\frac{M}{4}s'^2|\xi| -\frac{M}{4(1+\delta)}|\xi|(|x-x'-s'\mathrm{Re}\,\psi(x',t',s')|^2 \\
& +|t-t'|^2)+|\xi|O(|x'|s'^2+|t'|s'^2+ |s'|^3).
\end{split}
\end{equation}
We choose $r$ and $|s_{0}|$ small enough so that when $(x',t') \in \mathrm{supp}(\eta)$ and $|s'| \leq |s_{0}|, (\xi,\tau) \in \Gamma,$ (\ref{eqn113}) will yield,

\begin{equation}\label{eqn114}
\mathrm{Re}\, Q(x,t,x',t',\xi,\tau, s') \leq  -\frac{M}{8}s'^2|\xi| -\frac{M}{4(1+\delta)}|\xi|(|x-x'-s'\mathrm{Re}\,\psi(x',t',s')|^2 \\ +|t-t'|^2).
\end{equation}
From (\ref{eqn114}), it follows that the first integral on the right in (\ref{eqn18}) (at level $s'=s_{0}$) decays exponentially in $\xi$ and hence there are constants $C_{1},C_{2} >0$ such that for $(\xi,\tau) \in \Gamma,$

\begin{equation} \label{eqn115}
\left|\int_{\mathbb{R}^{n_{0}+1}} q(x,t,x',t',\xi,\tau,s_{0})dZ(x',t',s_{0})\right| \leq C_{1}e^{-C_{2}|(\xi,\tau)|}
\end{equation}

Consider next the second integral on the right in (\ref{eqn18}). To estimate it, we use equation (\ref{eqn17}) which is a sum of two kinds of terms. The first kind consists of terms involving $L_{1}(Z_{j}), L_{1}(Q)$ and $L_{1}u$(recall that $L_{1}u=Lu=0$) and these terms can be bounded by constant multiples of
$$|\xi||s'|^m e^{\mathrm{Re}\, Q(x,t,x',t',\xi,\tau,s')}, \forall m\geq 1,$$
and so using (\ref{eqn114}) which implies that
$$\mathrm{Re}\,Q(x,t,x',t',\xi,\tau,s') \leq -\frac{M}{8} s'^2|\xi|,$$
the integrals of such terms decay rapidly for $(\xi,\tau) \in \Gamma.$ The second type of terms involve derivatives
of $\eta(x,t)$ and hence $|x'|^2 +t'^2 \geq r^2$ in the domains of integration.
Therefore, if we choose $0 < |s_{0}| << r,$ we can get $\lambda >0$ such that for $(x,t)$ near $(0,0)$ and $(\xi,\tau) \in \Gamma,$ (\ref{eqn114}) will lead to,
$$\mathrm{Re}\,Q(x,t,x',t',\xi,\tau,s') \leq -\lambda|\xi|,~\text{when}~|x'|^2 +t'^2 \geq r^2.$$
The latter leads to an exponential decay in $(\xi,\tau) \in \Gamma$ for $(x,t)$ near $(0,0)$ for the corresponding integrals. We conclude that there exists a neighborhood $W$ of $(0,0)$ in $(x,t)$ space and an open conic neighborhood $\Gamma$ of $(\xi^{0},0)$ in $\mathbb{R}^{n_{0}+1}$ such that for $\forall (x,t) \in W, (\xi,\tau) \in \Gamma, \forall m=1,2,\cdots,$ there exists $C_{m} >0$ satisfying

$$\left|\int_{\mathbb{R}^{n_{0}+1}} e^{\sqrt{-1}[\xi(x-x')+\tau(t-t')]-K|(\xi,\tau)|(|x-x'|^2 +|t-t'|^2)} \eta(x',t')u(x',t')dx'dt'\right|$$

$$=\left|\int_{\mathbb{R}^{n_{0}+1}}q(x,t,x',t',\xi,0)dx'dt'\right| \leq \frac{C_{m}}{(1+|\xi|+|\tau|)^m}.$$
By Theorem 2.1 in [BH] (see also [T] and the proof of Lemma V.5.2 in [BCH]), we conclude that
$$(\xi^{0},0) \not \in WF(u)|_{0}.$$
Suppose now the Levi form $\mathcal{L}_{\sigma}$ at every $\sigma \in T_{p}^{0}$ has a nonzero eigenvalue. Define
$$S=\{\sigma \in T_{p}^0:\mathcal{L}_{\sigma}(v)\geq 0,~~\forall v \in \mathcal{V}_{p} \}. $$

The set $S$ is conic, closed and convex. If $\xi \in S,$ and $\xi \neq 0,$ then by hypothesis $\mathcal{L}_{\xi}$ has at least one positive eigenvalue and hence $-\xi \not \in S.$ Since $\xi \not \in WF(u),$ whenever $\mathcal{L}_{\xi}$ has at least one negative eigenvalue, it follows that $WF(u) \subset S,$ for every CR function near the point $p.$
\end{proof}

\begin{example} Let $M=\{(z_1,z_2)\in \Bbb C^2:\,\text{Im}\,z_2=|z_1|^{2m}\}$ where $m$ is a positive integer and let $M'=\{(z_1,z_2)\in \Bbb C^2:\,\text{Im}\,z_2=|z_1|^{2}\}$. Then the map  $H(z_1,z_2)=(z_1^m, z_2)$ is $1-$nondegenerate at the points where $z_1\neq 0$, and $m-$nondegenerate at all the other points. When $m>1$, $M$ itself is $1-$nondegenerate at the points where $z_1\neq 0$ while when $z_1=0$,  it is not $l-$nondegenerate for any $l \geq 0$. (The case $m=1$ appeared in [La1]. See also [K]).

\end{example}
\begin{example} Let $M=\{(z_1,z_2)\in \Bbb C^2:\,\text{Im}\,z_2=|z_1|^{2}\}$ and $M'=\{(w_1,w_2,w_3,w_4)\in \Bbb C^4:\,\text{Im}\,w_4=|w_1|^{2}+|w_2|^2-|w_3|^2\}$. For any odd positive integer $m\geq 3$, define $H_m(z_1,z_2):M\rightarrow M$ by $H_m(z_1,z_2)=(z_1, z_2^{\frac{m}{2}},z_2^{\frac{m}{2}},z_2)$ where we have used a branch of the square root. $H_m$ is a CR mapping and it is the boundary value of a holomorphic map defined on a side of $M$. $H_m$ is a diffeomorphism. $H_m$ is not smooth and so for each positive integer $k$, there is $m$ such that $H_m$ is in $C^k$ but it is not $k-$nondegenerate.
\end{example}

\begin{example} Let $M=\{(z_1,z_2)\in \Bbb C^2:\,\text{Im}\,z_2=|z_1|^{2}\}$ and $M'=\{(w_1,w_2)\in \Bbb C^3:\,\text{Im}\,w_3=|w_1|^{2}-|w_2|^2\}$. For any positive integer $m$, let $f:M\rightarrow \Bbb C$ be a CR function of class $C^m$ which is not smooth on any open subset of $M$ (see [BX] for an example of such). Define $H_m:M\rightarrow M'$ by $H_m(z_1,z_2)=(f(z_1,z_2),f(z_1,z_2),0)$. $H_m$ is a CR mapping of class $C^m$ which is not smooth on any open subset of $M$. Note that $H_m$ is not $k-$nondegenerate for any $k$.
\end{example}

\section{Proof of Theorem 2.3}
We begin by recalling the following ``almost holomorphic" version of the
implicit function theorem from [La1]:
\begin{theorem}\label{thimp}
Let $U \subset \mathbb{C}^{N}$ be open, $0 \in U,$ $A \in \mathbb{C}^{p},$ and $Z=(Z_{1},\cdots,Z_{N})$ be the coordinates in $\mathbb{C}^N,$ $W$ the coordinates in $\mathbb{C}^p.$ Let $F: U \times \mathbb{C}^{p}\rightarrow \mathbb{C}^{N} $ be smooth in the first $N$ variables and a polynomial in the last variables. Assume that $F(0,A)=0$ and $F_{Z}(0,A)$ is invertible. Then there exists a neighborhood $U' \times V'$ of $(0,A)$ and a smooth map $\psi=(\psi_{1},\cdots,\psi_{N}): U' \times V'\rightarrow \mathbb{C}^{N}$ with $\psi(0,A)=0,$ such that if $F(Z,\overline{Z},W)=0$ for some $(Z,W) \in U'\times V',$ then $Z=\psi(Z,\overline{Z},W).$ Furthermore, for every multiindex $\alpha,$ and each $j,$ $1\leq j \leq N,$
\begin{equation}\label{e2.4}
D^{\alpha}\frac{\partial \psi_{j}}{\partial Z_{i}}(Z,\overline{Z},W)=0, \,\,  \,\,1 \leq i \leq N,
\end{equation}
if $Z=\psi(Z,\overline{Z},W),$ and $\psi$ is holomorphic in $W.$ Here $D^{\alpha}$ denotes the derivative in all real variables.
\end{theorem}

Given the abstract  CR manifold $(M,\mathcal{V})$ of CR dimension $n$ and CR codimension $d,$ we will use local coordinates $(x,y,s) \in \mathbb{R}^n \times \mathbb{R}^n \times \mathbb{R}^d$ that vanish at $p_{0} \in M.$
We will write $z=(z_{1},\cdots,z_{n})$ where $z_{j}=x_{j}+\sqrt{-1}\,y_{j}$ for $j=1,\cdots,n.$ In a neighborhood $W$ of $0,$ we may assume that a basis of $\mathcal{V}$ is given by $\{L_{1},\cdots,L_{n}\}$ where
$$L_{i}=\frac{\partial}{\partial \overline{z}_{i}} + \sum_{j=1}^{n} a_{ij}(x,y,s)\frac{\partial}{\partial z_{j}}+ \sum_{l=1}^d b_{il}(x,y,s)\frac{\partial}{\partial s_{l}},\,\, 1 \leq i \leq n,$$
the $a_{ij}$ and $b_{il}$ are smooth and $a_{ij}(0)=0=b_{il}(0), \forall i,j,l$ (see for example [BCH], equation I.19).
 In these coordinates, at the origin,  the characteristic set
$$T_{0}^0=\{(\xi,\eta,\sigma) \in \mathbb{R}^{n} \times \mathbb{R}^{n} \times \mathbb{R}^d: \xi=\eta=0\}.$$
By assumption, there is an acute open convex cone $\Gamma \subset \mathbb{R}^d$ such that
$$WF(H_{j})|_{0} \subset \{(0,0,\sigma): \sigma \in \Gamma \}, \forall j=1,\cdots,N'.$$

Let $\phi \in C_{0}^{\infty}(W)$ whose support is sufficiently small and $\phi \equiv 1$ in a neighborhood of the origin. For each $j=1, \cdots, N',$ by Fourier's inversion formula,
\begin{equation} \label{eqn21}
\begin{split}
\phi (x,y,s)H_{j}(x,y,s)=&\int_{\mathbb{R}^{2n+d}} e^{2\pi\sqrt{-1}(x \cdot \xi +y \cdot \eta + s \cdot \sigma)}\widehat{\phi H_{j}}(\xi,\eta,\sigma)d \sigma d \eta d \xi \\
=& \int_{A}e^{ 2\pi\sqrt{-1}(x \cdot \xi +y \cdot \eta + s \cdot \sigma)}\widehat{\phi H_{j}}(\xi,\eta, \sigma)d\sigma d \eta d\xi \\
+& \int_{\mathbb{R}^{2n+d} \setminus A} e^{2\pi\sqrt{-1}(x \cdot \xi +y \cdot \eta + s \cdot \sigma)}\widehat{\phi H_{j}}(\xi,\eta,\sigma)d \sigma d \eta d \xi \\
=& I^{j}(x,y,s)+J^{j}(x,y,s)
\end{split}
\end{equation}
where $A=\{(\xi,\eta,\sigma) \in \mathbb{R}^{n} \times \mathbb{R}^{n} \times \mathbb{R}^{d}: \sigma \not \in \Gamma\}.$

Since $WF(H_{j})|_{0} \subset \{(0,0,\sigma): \sigma \in \Gamma\},$ if the support of $\phi$ is sufficiently small,
for every $m=1,2,\cdots,$ there exists a constant $C_m>0$ such that
$$|\widehat{\phi H_{j}}(\xi,\eta,\sigma)|\leq \frac{C_{m}}{(1+|\xi|+|\eta|+|\sigma|)^m}, \,\,\forall (\xi,\eta,\sigma) \in A.$$
It follows that $I^{j}(x,y,s)$ is $C^{\infty}$ on $\mathbb{R}^{2n+d}.$
Write
$$J^{j}(x,y,s)=\int_{B_{1}}e^{2\pi\sqrt{-1}(x \cdot \xi + y \cdot \eta +s\cdot \sigma)}\widehat{\phi H_{j}}(\xi,\eta, \sigma)d\sigma d \eta d\xi + \int_{B_{2}}e^{2\pi\sqrt{-1}(x \cdot \xi + y \cdot \eta +s\cdot \sigma)}\widehat{\phi H_{j}}(\xi,\eta, \sigma)d\sigma d \eta d\xi $$
where
$$B_{1}=\{(\xi,\eta,\sigma):|\xi|^{2}+|\eta|^2 \leq 1, \,\,\sigma \in \overline{
\Gamma} \},$$
$$B_{2}=\{(\xi,\eta,\sigma):|\xi|^{2}+|\eta|^{2} \geq 1, \sigma \in \overline{\Gamma} \}.$$

Observe that since $T_{0}^{0} \cap B_{2} =\emptyset,$ for any CR function $u$ near the origin, $WF(u)|_{0} \cap B_{2}=\emptyset.$
Moreover,
$$B_{2} \cap \{(\xi,\eta,\sigma):|\xi|^{2}+|\eta|^{2}+|\sigma|^{2}=1\}$$
is a compact set. It follows that for each $m=1,2,\cdots,$ we can get $C_{m}'>0$ such that
\begin{equation}\label{eqn22}
|\widehat{\phi H_{j}}(\xi,\eta,\sigma)| \leq \frac{C_{m}'}{(1+|\xi|+|\eta|+|\sigma|)^m},\,\, \forall (\xi,\eta,\sigma) \in B_{2}
\end{equation}
It follows that
$${F}_{2}^{j}(x,y,s)= \int_{B_{2}} e^{2\pi\sqrt{-1}(x \cdot \xi +y \cdot \eta + s\cdot \sigma)} \widehat{\phi H_{j}}(\xi,\eta,\sigma) d \sigma d\eta d\xi$$
is $C^{\infty}$ on $\mathbb{R}^{2n+d}.$

Since $\Gamma$ is an acute cone, there is $\sigma^{0} \in \mathbb{R}^d$ such that $\sigma^{0}\cdot \sigma >0, \forall \sigma \in \Gamma.$ We may assume that for some conic neighborhood $\Gamma_{1}$ of $\sigma$ and $C_{0}>0,$
\begin{equation}\label{eqn23}
v\cdot \sigma \geq C_{0}|v||\sigma|,~\forall v \in \Gamma_{1},\,\sigma \in \Gamma.
\end{equation}
For $t \in \Gamma_{1},$ we define
$$F_{1}^{j}(x,y,s,t)=\int_{B_{1}}e^{2\pi\sqrt{-1}(x \cdot \xi+ y \cdot \eta +(s +\sqrt{-1}t)\cdot \sigma)} \widehat{\phi H_{j}}(\xi,\eta,\sigma)d \sigma d \eta d \xi.$$
Since $\widehat{\phi H_{j}}$ has a polynomial growth, for some $C_{1},M >0,$
\begin{equation}\label{eqn24}
|\widehat{\phi H_{j}}(\xi,\eta,\sigma)| \leq C_{1}(1+|\sigma|)^{M},\,\,\forall (\xi,\eta,\sigma) \in B_{1}.
\end{equation}
Therefore, using (\ref{eqn23}) and (\ref{eqn24}), we get,
\begin{equation}
|F_{1}^{j}(x,y,s,t)| \leq C'_{1} \int_{\mathbb{R}^{d}} e^{-C_{0}|t||\sigma|}(1+|\sigma|)^{M} d \sigma \leq \frac{C_{2}}{|t|^{M+d+1}},~~t \in \Gamma_{1},~\text{for some}~C'_{1},C_{2} >0.
\end{equation}
Moreover, for all multiindices $\alpha,\beta \in \mathbb{N}^{n}, \gamma \in \mathbb{N}^d ,$
\begin{equation} \label{eqn26}
|\partial_{x}^{\alpha}\partial_{y}^{\beta}\partial_{s}^{\gamma}F_{1}^{j}(x,y,s,t)| \leq \frac{C}{|t|^{M+d+1+|\gamma|}},
\end{equation}
for some $C >0$ when $t \in \Gamma_{1}.$

When $t \in \Gamma_{1},$
\begin{equation}\label{eqn27}
\overline{\partial}_{w_{\nu}}F_{1}^{j}(x,y,s,t)=0,~\text{for}~1 \leq \nu \leq d,
\end{equation}
where $\overline{\partial}_{w_{\nu}}=\frac{1}{2}(\frac{\partial}{\partial s_{\nu}}+ \sqrt{-1}\frac{\partial}{\partial t_{\nu}}).$

Define
$$F_{2}^{j}(x,y,s,t)=\int_{B_{2}} e^{2\pi\sqrt{-1}(x \cdot \xi+ y \cdot \eta +(s +\sqrt{-1} t) \cdot \sigma)}\widehat{\phi H_{j}}(\xi,\eta,\sigma)d \xi d\eta d\sigma,$$
for $t \in \Gamma_{1}.$ By (\ref{eqn22}), $F_{2}^{j}$ is $C^{\infty}$ up to $t=0,$
and
\begin{equation}\label{eqn28}
\overline{\partial}_{w_{\nu}}F_{2}^{j}(x,y,s,t)=0,~\text{for}~1 \leq \nu \leq d,\, t \in \Gamma_{1}.
\end{equation}

Since $I^{j}(x,y,s)$ is $C^{\infty}$ and bounded, we can find a bounded $C^{\infty}$ function $F_{0}^{j}(x,y,s,t)~(|t|~\text{small})$ such that
\begin{equation}\label{eqn29}
F_{0}^{j}(x,y,s,0)=I^{j}(x,y,s),~\text{and}~\overline{\partial}_{w_{\nu}} F_{0}^{j}(x,y,s,t)=O(|t|^{l}),~\forall \nu=1,\cdots,d, \forall l=1,2,3,\dots
\end{equation}

Let $\varphi (x,y,s) \in C_{0}^{\infty}(W)$ such that its support is contained in a neighborhood of the origin where $\phi \equiv 1.$ By Parseval's formula,
\begin{equation}\label{eqn210}
\begin{split}
\lim_{t\rightarrow 0, t  \in \Gamma_{1}} \int_{\mathbb{R}^{2n+d}} F_{0}^{j}(x,y,s,t) \varphi(x,y,s)dx dy ds = & \int_{\mathbb{R}^{2n+d}} I^{j}(x,y,s)\varphi (x,y,s)dx dy ds\\ = & \int_{\mathbb{R}^{2n+d}} \widehat{I^{j}}(\xi,\eta,\sigma)\widehat{\varphi}(-\xi,-\eta,-\sigma) d\xi d\eta d\sigma\\ = & \int_{A}\widehat{\phi H_{j}}(\xi,\eta,\sigma)\widehat{\varphi}(-\xi,-\eta,-\sigma) d\xi d\eta d\sigma
\end{split}
\end{equation}
Likewise, since $F_{2}^{j}$ is $C^{\infty}$ and bounded,
\begin{equation}\label{eqn211}
\int_{\mathbb{R}^{2n+d}} F_{2}^{j}(x,y,s) \varphi(x,y,s)dx dy ds= \int_{B_{2}} \widehat{\phi H_{j}}(\xi,\eta,\sigma)\widehat{\varphi}(-\xi,-\eta,-\sigma)d\xi d\eta d\sigma.
\end{equation}

For $t \in \Gamma_{1},$ using (\ref{eqn23}), we have,
\begin{equation}
\begin{split}
&\int_{\mathbb{R}^{2n+d}} F_{1}^{j}(x,y,s,t)\varphi (x,y,s)dx dy ds\\= & \int_{B_{1}}(\int_{\mathbb{R}^{2n+d}} e^{2\pi\sqrt{-1}(x \cdot \xi+y \cdot \eta +s \cdot \sigma)} \varphi(x,y,s)dx dy ds)e^{-t \cdot \sigma} \widehat{\phi H_{j}}(\xi,\eta, \sigma)d \xi d\eta d\sigma\\
= & \int_{B_{1}} \widehat{\varphi}(-\xi,-\eta,-\sigma)e^{-t \cdot \sigma} \widehat{\phi H_{j}}(\xi,\eta,\sigma)d \xi d\eta d\sigma,
\end{split}
\end{equation}
and hence
\begin{equation}\label{eqn212}
\lim_{t \rightarrow 0,~t \in \Gamma_{1}} \int_{\mathbb{R}^{2n+d}} F_{1}^{j}(x,y,s,t)\phi(x,y,s)dx dy ds= \int_{B_{1}} \widehat{\varphi}(-\xi,-\eta,-\sigma) \widehat{\phi H_{j}}(\xi,\eta,\sigma)d \xi d\eta d\sigma
\end{equation}

Let $F^{j}(x,y,s,t)=F_{0}^{j}(x,y,s,t)+F_{1}^{j}(x,y,s,t)+F_{2}^{j}(x,y,s,t)$ for $t \in \Gamma_{1}.$
From (\ref{eqn210}),(\ref{eqn211}) and (\ref{eqn212}),
\begin{equation}
\begin{split}
\lim_{t \rightarrow 0,~t \in \Gamma_{1}} \int_{\mathbb{R}^{2n+d}} F^{j}(x,y,s,t)\varphi(x,y,s)dx dy ds=& \int_{\mathbb{R}^{2n+d}} \widehat{\varphi}(-\xi,-\eta,-\sigma) \widehat{\phi H_{j}}(\xi,\eta,\sigma)d \xi d\eta d\sigma\\
=&\int_{\mathbb{R}^{2n+d}} \phi(x,y,s) H_{j}(x,y,s) \varphi(x,y,s)dx dy ds.
\end{split}
\end{equation}
Therefore, in a neighborhood of the origin, in the distribution sense,
\begin{equation}\label{eqn213}
\lim_{t \rightarrow 0,~ t \in \Gamma_{1}} F^{j}(x,y,s,t)=H_{j}(x,y,s).
\end{equation}
For $t \in \Gamma_{1}$ small, from (\ref{eqn26})-(\ref{eqn29}), we have:
for $(x,y,s)$ near $0,$ given $\alpha,\beta,\gamma,$ there exists $C_{1}>0$ such that for some $\lambda>0,$
\begin{equation}\label{eqn214}
|\partial_{x}^{\alpha}\partial_{y}^{\beta}\partial_{s}^{\gamma}F^{j}(x,y,s,t)| \leq \frac{C_{1}}{|t|^\lambda},~\text{and}
\end{equation}
\begin{equation}\label{eqn215}
\partial_{x}^{\alpha}\partial_{y}^{\beta}\partial_{s}^{\gamma}\overline{\partial}_{w_{\nu}}F^{j}(x,y,s,t)=O(|t|^{l}),
\,\,\forall l \geq 1,~\forall \nu=1,\cdots,d.
\end{equation}

For the rest of the proof, we follow the argument of claim 3 in [La1]. We may assume that $H(0)=0 \in M'.$
Let $\rho=(\rho_{1},\cdots,\rho_{d'})$ be defining functions for $M'$ near $0.$ For $\alpha \in \mathbb{N}^{n}$ a multiindex, recall that $L^{\alpha}=L_{1}^{\alpha_{1}}\cdots L_{n}^{\alpha_{n}}.$

Set $F(x,y,s,t)=(F^{1}(x,y,s,t),\cdots,F^{N'}(x,y,s,t)),t \in \Gamma_{1}.$ As in [La1], there are smooth functions
$\Psi_{\mu,\alpha}(Z',\overline{Z'},W)$ for $|\alpha| \leq k_{0}, 1\leq \mu \leq d',$ defined in a neighborhood of $\{0\} \times \mathbb{C}^{K(k_{0})}$ in $\mathbb{C}^{N'} \times \mathbb{C}^{K(k_{0})},$ polynomial in $W,$ such that
\begin{equation}\label{eqn216}
L^{\alpha} \rho_{\mu}(H(z,s),\overline{H(z,s)})=\Psi_{\mu,\alpha}(H(z,s),\overline{H(z,s)},(L^{\beta}\overline{H}(z,s))_{|\beta| \leq k_{0}}),
\end{equation}
and
\begin{equation}\label{eqn217}
L^{\alpha}\rho_{\mu,Z'}(H,\overline{H})|_{0}=\Psi_{\mu,\alpha,Z'}(0,0,(L^{\beta} \overline{H}(0,0))_{|\beta| \leq k_{0}}).
\end{equation}
Here $K(k_{0})=N'|\{\beta:|\beta| \leq k_{0}\}|.$ Equation (\ref{eqn217}) and the $k_{0}-$nondegeneracy assumption on the map $H$ allows us to get $(\alpha^{1},\cdots,\alpha^{N'}),(\mu_{1},\cdots,\mu_{N'}) \in \mathbb{N}^{N'}$ and a smooth function $\psi(Z',\overline{Z'},W)=(\psi_{1},\cdots,\psi_{N'}),$ which is holomorphic in $W,$ such that with
\[
\Psi=(\Psi_{\mu_{1},\alpha^{1}},\cdots,\Psi_{\mu_{N'},\alpha^{N'}}),
\]
if $\Psi(Z',\overline{Z'},W)=0,$ then $Z'=\psi(Z',\overline{Z'},W).$ Moreover, with $Z'=(z'_{1},\cdots,z'_{N'}),$ we have,
\begin{equation}\label{eqn218}
D^{\alpha}\frac{\partial \psi_{j}}{\partial z'_{i}}(Z',\overline{Z'},W)=0,~\forall i=1,\cdots,N',\,j=1,\cdots,N',
\end{equation}
whenever $Z'=\psi(Z',\overline{Z'},W).$ In particular, since $\Psi_{l,\alpha}(H(z,s),\overline{H}(z,s),(L^{\beta}\overline{H}(z,s))_{|\beta|\leq k_{0}})=0,$ we have,

\begin{equation}\label{eqn219}
H_{j}(z,s)=\psi_{j}(F(z,s,0),\overline{F}(z,s,0),(L^{\beta}\overline{F}(z,s,0))_{|\beta|\leq k_{0}}),
\forall j=1,\cdots,N'.
\end{equation}
Recall that for $i=1,\cdots,n,$
$$L_{i}=\frac{\partial}{\partial \overline{z}_{i}}+\sum_{j=1}^{n}a_{ij}(x,y,s)\frac{\partial}{\partial z_{j}} + \sum_{l=1}^{d} b_{il}(x,y,s)\frac{\partial}{\partial s_{l}}.$$
Let
$$M_{i}=\frac{\partial}{\partial \overline{z}_{i}} + \sum_{j=1}^{n} A_{ij}(x,y,s,t)\frac{\partial}{\partial z_{j}} +\sum_{l=1}^{d} B_{il}(x,y,s,t)\frac{\partial}{\partial s_{l}}, 1 \leq i \leq n,$$
where the $A_{ij}$ and $B_{il}$ are smooth extensions of the $a_{ij}$ and $b_{il}$ satisfying
\begin{equation}\label{eqn220}
\overline{\partial}_{w_{\nu}}A_{ij}(x,y,s,t),\overline{\partial}_{w_{\nu}}B_{il}(x,y,s,t)=O(|t|^{m}),~\forall \nu=1,\cdots,d,~\forall m=1,2,\cdots.
\end{equation}

Now define
$$g_{j}(z,s,t)=\psi_{j}(F(z,s,-t),\overline{F}(z,s,-t),(M^{\beta}\overline{F}(z,s,-t))_{|\beta| \leq k_{0}}),$$
for $j=1,\cdots,N'$ and for $t \in -\Gamma_{1},|t|$ small. Using (\ref{eqn215}),(\ref{eqn218}) and (\ref{eqn220}), we conclude that, when $(z,s)$ is near the origin in $\mathbb{C}^{n} \times \mathbb{R}^{d}$ and $t \in -\Gamma_{1} \,(|t| ~\text{small}),$ for any $\alpha,\beta,\gamma$ multiindices, there is $C>0$ such that
\begin{equation}\label{eqn221}
|D_{x}^{\alpha}D_{y}^{\beta}D_{s}^{\gamma}g_{j}(z,s,t)| \leq \frac{C}{|t|^{\lambda}}~\text{for some}\,\lambda >0.
\end{equation}
and
\begin{equation}\label{eqn222}
D_{x}^{\alpha}D_{y}^{\beta}D_{s}^{\gamma}\overline{\partial}_{w_{\nu}}g_{j}(z,s,t)=O(|t|^{m}), \forall m=1,2,\cdots,\nu=1,\cdots,d.
\end{equation}
From (\ref{eqn219}), we know that,
\begin{equation}\label{eqn223}
H_{j}(z,s)=\lim_{t \rightarrow 0,~t \in -\Gamma_{1}}g_{j}(z,s,t), \forall j=1,\cdots,N'.
\end{equation}
By Theorem V.3.7 in [BCH], it follows that $\mathrm{WF}(H_{j})|_{0} \cap \Gamma =\emptyset.$ Since by assumption $WF(H_{j})|_{0} \subset \Gamma,$ we conclude that $H$ is $C^{\infty}$ near the origin.

\section{Proof of Theorem 2.5}
Fix any $p \in M,$ and assume $p'=F(p)=0.$  Since $M'$ is strictly pseudoconvex, we may assume that there is a
neighborhood $G$ of $0$ in $\mathbb{C}^{n+k},$ and a local defining
function $\rho$ of $M'$ in $G$ such that
$$M' \cap G=\{Z' \in G:\rho(Z',\overline{Z'})=0\},$$
where
$\rho(Z',\overline{Z'})=-v'+\sum_{j=1}^{n+k-1}|z'_{j}|^2+\phi^{*}(Z',\overline{Z'})$. Here $Z'=(z'_{1},\cdots,z'_{n+k})$ are the coordinates of
$\mathbb{C}^{n+k},$ $z'_{n+k}=u'+\sqrt{-1} v'$ and
$\phi^{*}(Z',\overline{Z'})=O(|Z'|^3)$ is a real-valued smooth
function on $G$.
Note that $\mathrm{rank}_{l}(F,p)$ is a lower semi-continuous
integer-valued function on $M$ for each $1 \leq l \leq k.$ For any $p \in M,$

$$\mathrm{rank}_{0}(F,p) \leq \mathrm{rank}_{1}(F,p) \leq \cdots \leq \mathrm{rank}_{k}(F,p).$$

We next recall  some basic properties of the rank of $F.$ Write
$F=(F_{1},\cdots,F_{n+k}).$ Since $F(M) \subset M',$
we have
\begin{equation}
\rho(F,\overline{F})=-\frac{F_{n+k}-\overline{F_{n+k}}}{2\sqrt{-1}}+F_{1}\overline{F_{1}}+\cdots+
F_{n+k-1}\overline{F_{n+k-1}}+\phi^{*}(F,\overline{F})=0,
\end{equation}
on $M$ near $p.$ Applying
$L_{1},\cdots,L_{n}$ to the above equation,
we get
\begin{equation}\label{eqlrhoz}
\frac{L_{j}\overline{F_{n+k}}}{2\sqrt{-1}}+F_{1}L_{j}\overline{F_{1}}+\cdots+F_{n+k-1}
L_{j}\overline{F_{n+k-1}}+L_{j}\phi^{*}(F,\overline{F})=0,\,1 \leq j \leq n,
\end{equation}

\begin{equation}\label{eqlberhoz}
\frac{L^{\alpha}\overline{F_{n+k}}}{2\sqrt{-1}}+F_{1}L^{\alpha}\overline{F_{1}}+\cdots+F_{n+k-1}
L^{\alpha}\overline{F_{n+k-1}}+L^{\alpha}\phi^{*}(F,\overline{F})=0,
\end{equation}
on $M$ near $p$ for any multiindex $1 \leq |\alpha| \leq k.$ Therefore, on $M$ near  $p,$
\begin{equation}\label{eqrhoz}
\rho_{Z'}(F,\overline{F})=
(\overline{F_{1}}+\phi^{*}_{z'_{1}}(F,\overline{F}),\cdots,\overline{F_{n+k-1}}+\phi^{*}_{z'_{n+k-1}}(F,\overline{F}),
\frac{\sqrt{-1}}{2}+\phi^{*}_{z'_{n+k}}(F,\overline{F})),
\end{equation}
and for any multiindex $1 \leq |\alpha| \leq k,$
\begin{equation}
L^{\alpha}\rho_{Z'}(F,\overline{F})=(L^{\alpha}(\overline{F_{1}}+\phi^{*}_{z'_{1}}),
\cdots, L^{\alpha}(\overline{F_{n+k-1}}+\phi^{*}_{z'_{n+k-1}}),L^{\alpha}\phi^{*}_{z'_{n+k}}).
\end{equation}

\begin{lemma}\label{lerankc}
With the assumption of Theorem \ref{thm}, for any $p \in M,$ we have
$\mathrm{rank}_{0}(F,p)=1,~\mathrm{rank}_{1}(F,p)=n+1,$ and thus $\mathrm{rank}_{l}(F,p) \geq n+1,$ for $1 \leq l \leq k.$
\end{lemma}
\begin{proof}
Assume that $F(p)=0.$ Note that $\phi^{*}_{z'_{i}}(F,\overline{F})|_{p}=0,$ for all $1 \leq i \leq n+k.$ Equation $(4.4)$ shows that
$\mathrm{rank}_{0}(F,p)=1.$ By assumption,  $dF: \mathcal{V}_{p} \rightarrow T^{(0,1)}_{0}M'$ is injective.
 By plugging $Z=p$ in equation $(4.2)$, we get $L_{i}\overline{F}_{n+k}(p)=0$ for each
 $1 \leq i \leq n.$ Since $\{L_{1}, L_{2},\cdots,L_{n}\}$ is a local basis
  of $\mathcal{V}$ near $p,$ we conclude that the rank of the matrix
   $(L_{i}\overline{F}_{l})_{1 \leq i \leq n, 1 \leq l \leq n+k-1}$ is $n.$
   Without loss of generality, we assume that

$$\left|\begin{array}{ccccc}
    L_{1}\overline{F}_{1} & . & . & . & L_{1}\overline{F}_{n} \\
    . & . & . & . & . \\
    . & . & . & . & . \\
    . & . & . & . & . \\
    L_{n}\overline{F}_{1} & . & . & . & L_{n}\overline{F}_{n}
  \end{array}\right| \neq 0~\text{at}~p.
$$
Notice that~$\phi^{*}_{z'_{1}}|_{p}=\phi^{*}_{z'_{2}}|_{p}=\cdots=\phi^{*}_{z'_{n+k}}|_{p}=0,$ $L_{j}\phi^{*}_{z'_{1}}|_{p}=L_{j}\phi^{*}_{z'_{2}}|_{p}=\cdots=L_{j}\phi^{*}_{z'_{n+k}}|_{p}=0,$~for all $1 \leq j \leq n.$ Thus $\mathrm{rank}_{1}(F,p)=n+1.$ Consequently, $\mathrm{rank}_{l}(F,p) \geq n+1$ for $1 \leq l \leq k$ for any $p\in M$.
\end{proof}

To simplify the notations, let
$$a_{i}(Z,\overline{Z})=\overline{F}_{i}+\phi^{*}_{z'_{i}}(F,\overline{F}), \,1  \leq i \leq n+k-1,$$
$$a_{n+k}(Z,\overline{Z})=\frac{\sqrt{-1}}{2}+\phi^{*}_{z'_{n+k}}(F,\overline{F}),$$
$${\bf{a}}(Z,\overline{Z})=(a_{1},\cdots,a_{n+k}).$$
Then
$$\rho_{Z'}(F,\overline{F})={\bf{a}}=(a_{1},\cdots,a_{n+k-1},a_{n+k}),$$
$$L^{\alpha}\rho_{Z'}(F,\overline{F})=L^{\alpha}{\bf{a}}=(L^{\alpha}a_{1},\cdots,L^{\alpha}a_{n+k-1}
,L^{\alpha}a_{n+k})$$
for any multiindex $0 \leq |\alpha| \leq k.$ Recall that
$$\mathrm{rank}_{l}(F,p)=\mathrm{dim}_{\mathbb{C}}(\mathrm{Span}_{\mathbb{C}}\{L^{\alpha}{\bf{a}}(Z,\overline{Z})|_{p}:0 \leq |\alpha| \leq l\}).$$

The following normalization will be applied later in this section.

\begin{lemma}\label{lenorm}
Let $M, M', F$ be as in Theorem 2.5 and $p=0\in M$. Assume
$\mathrm{rank}_{l}(F,p)=N_{0},$ for some $1 \leq l \leq k,n+1 \leq N_{0} \leq n+k.$
Then there exist multiindices
$\{\beta_{n+1},\cdots,\beta_{N_{0}-1}\}$ with $1 < |\beta_{i}|
\leq l$ for all $i,$ such that after a linear biholomorphic change
of coordinates in $\mathbb{C}^{n+k}: \widetilde{Z}=Z'A^{-1},$ where
$A$ is a unitary $(n+k) \times (n+k)$ matrix, and $\widetilde{Z}$
denotes the new coordinates, the following hold:
\begin{equation}\label{eqnorm}
\widetilde{\bf{a}}|_{p}=(0,\cdots,0,\frac{\sqrt{-1}}{2}),
\left(\begin{array}{c}
  L_{1}\widetilde{\bf{a}}|_{p} \\
  \cdots \\
  L_{n}\widetilde{\bf{a}}|_{p} \\
  L^{\beta_{n+1}}\widetilde{\bf{a}}|_{p}  \\
  \cdots \\
  L^{\beta_{N_{0}-1}}\widetilde{\bf{a}}|_{p}
\end{array}\right)=\left(\begin{array}{ccc}
                     {\bf{B}}_{N_{0}-1} & {\bf{0}} & {\bf{b}}
                   \end{array}\right).
\end{equation}
Here we write $\widetilde{\bf{a}}=\widetilde{\rho}_{\widetilde{Z}}(\widetilde{Z}(F),\overline{\widetilde{Z}(F)}),$ and $\widetilde{\rho}$ is a local defining function of $M'$ near $0$ in the new coordinates. Moreover,  ${\bf{B}}_{N_{0}-1}$ is an invertible $(N_{0}-1)\times (N_{0}-1)$ matrix, ${\bf{0}}$ is an $(N_{0}-1) \times(n+k-N_{0})$ zero matrix, and ${\bf{b}}$ is an $(N_{0}-1)-$dimensional column vector.
\end{lemma}

\begin{proof}
It follows from  Lemma \ref{lerankc} that
$$\{{\bf{a}},L_{1}{\bf{a}},\cdots,L_{n}{\bf{a}}\}|_{p}$$
is linearly independent. Extend it to a basis of $E_{l}(p),$ which
has dimension $N_{0}$ by assumption. That is, we choose
multiindices $\{\beta_{n+1},\cdots,\beta_{N_{0}-1}\}$ with $1 <
|\beta_{i}| \leq l$ for each $i,$ such that
$$\{{\bf{a}},L_{1}{\bf{a}},\cdots,L_{n}{\bf{a}},
L^{\beta_{n+1}}{\bf{a}},\cdots,L^{\beta_{N_{0}-1}}{\bf{a}}\}|_{p}$$
is linearly independent over $\mathbb{C}$.
We write $\widehat{\bf{a}}:=(a_{1},\cdots,a_{n+k-1}),$ that is, the first $n+k-1$ components of ${\bf{a}}.$
Notice that ${\bf{a}}(p)=(0,\cdots,0,\frac{\sqrt{-1}}{2}).$ Consequently,
$$\{L_{1}{\bf{\widehat{a}}},\cdots,L_{n}{\bf{\widehat{a}}},L^{\beta_{n+1}}{\bf{\widehat{a}}},\cdots,
L^{\beta_{N_{0}-1}}{\bf{\widehat{a}}}\}|_{p}$$
is linearly independent in $\mathbb{C}^{n+k-1}$. Let $S$ be the $(N_{0}-1)-$dimensional vector space spanned by them and let $\{T_{1},\cdots,T_{N_{0}-1}\}$ be an orthonormal basis of $S.$ Extend it to an orthonormal basis of $\mathbb{C}^{n+k-1}:\{T_{1},\cdots,T_{N_{0}-1},T_{N_{0}},\cdots,T_{n+k-1}\}$ and set
$$T=\left(\begin{array}{c}
       T_{1} \\
       \cdots \\
       T_{n+k-1}
     \end{array}\right)^{t},~~~~~
A=\left(\begin{array}{cc}
    T & {\bf{0}}^{t}_{n+k-1} \\
    {\bf{0}}_{n+k-1} & 1
  \end{array}\right).
$$
Here ${\bf{0}}_{n+k-1}$ is an $(n+k-1)-$dimensional zero row vector.
Next we make the following change of coordinates: $Z'=\widetilde{Z}A,~\text{or}~\widetilde{Z}=Z'A^{-1}.$
 The function  $\widetilde{\rho}(\widetilde{Z},\overline{\widetilde{Z}})=\rho(\widetilde{Z}A,\overline{\widetilde{Z}A})$
is a defining function of $M'$ near $0$ with respect to the new coordinates $\widetilde{Z}.$ By the chain rule,
\begin{equation}
\widetilde{\rho}_{\widetilde{Z}}(\widetilde{Z}(F),\overline{\widetilde{Z}(F)})
=\rho_{Z'}(F,\overline{F})A.
\end{equation}
For any multiindex $\alpha,$
\begin{equation}
L^{\alpha}\widetilde{\rho}_{\widetilde{Z}}(\widetilde{Z}(F),\overline{\widetilde{Z}(F)})
=L^{\alpha}\rho_{Z'}(F,\overline{F})A.
\end{equation}
In particular, at $p$, we get:

\begin{equation}
\widetilde{\bf{a}}|_{p}={\bf{a}}|_{p}A,~
\left(\begin{array}{c}
  L_{1}\widetilde{\bf{a}}|_{p} \\
  \cdots \\
  L_{n}\widetilde{\bf{a}}|_{p} \\
  L^{\beta_{n+1}}\widetilde{\bf{a}}|_{p}  \\
  \cdots \\
  L^{\beta_{N_{0}-1}}\widetilde{\bf{a}}|_{p}
\end{array}\right)=\left(\begin{array}{c}
  L_{1}{{\bf{a}}}|_{p} \\
  \cdots \\
  L_{n}{{\bf{a}}}|_{p} \\
  L^{\beta_{n+1}}{{\bf{a}}}|_{p}  \\
  \cdots \\
  L^{\beta_{N_{0}-1}}{{\bf{a}}}|_{p}
\end{array}\right)A.
\end{equation}
Furthermore from the definition of $A,$  in the new coordinates, equation $(\ref{eqnorm})$ holds and ${\bf{B}}_{N_{0}-1}$ is invertible.
\end{proof}

\begin{remark}
From the construction of $A$ in the proof of Lemma \ref{lenorm}, one
can see that in the new coordinates $\widetilde{Z}$, the following continues to hold: There is a neighborhood $G$ of $p'=0$ in $\mathbb{C}^{n+k},$
and a smooth real-valued function $\widetilde{\rho}$ in $G,$ such
that,
$$M' \cap G=\{\widetilde{Z} \in G:\widetilde{\rho}(\widetilde{Z},\overline{\widetilde{Z}})=0\}.$$
Moreover, $\widetilde{\rho}(\widetilde{Z}, \overline{\widetilde{Z}})=-\widetilde{v}+\sum_{j=1}^{n+k-1}|\widetilde{z}_{j}|^2
+\widetilde{\phi^{*}}(\widetilde{Z},\overline{\widetilde{Z}}), ~\text{where}~ \widetilde{Z}=(\widetilde{z}_{1},\cdots,\widetilde{z}_{n+k}),
\widetilde{z}_{n+k}=\widetilde{u}+\sqrt{-1}\widetilde{v}$ and $\widetilde{\phi^{*}}(\widetilde{Z}, \overline{\widetilde{Z}})=O(|\widetilde{Z}|^3)$ is a real-valued smooth function in $G.$ We will write the new coordinates as $Z$ instead of $\widetilde{Z}.$
\end{remark}

We will next prove some lemmas on the determinants of matrices.
\begin{lemma}\label{leal1}
For a general $n \times n $ matrix
$$B=\left(\begin{array}{cccccc}
    b_{11} & b_{12} & . & . & . & b_{1n} \\
    b_{21} & b_{22} & . & . & . & b_{2n} \\
    . & . & . & . & . & . \\
    . & . & . & . & . & . \\
    . & . & . & . & . & . \\
    b_{n1} & b_{n2} & . & . & . & b_{nn}
  \end{array}\right),
$$
where $b_{ij} \in \mathbb{C}$ for all $1 \leq i,j \leq n, n \geq 3,$ we have,

$
\left|\begin{array}{cc}
    B(\begin{array}{ccccccc}
        1 & 2 & . & . & . & n-2 & n-1  \\
        1 & 2 & . & . & . & n-2 & n-1
      \end{array}
    ) &~~~~~~B(\begin{array}{ccccccc}
            1 & 2 & . & . & . & n-2 & n-1 \\
            j_{1} & j_{2} & . & . & . & j_{n-2} & n
          \end{array}
    ) \\
    B(\begin{array}{ccccccc}
        i_{1} & i_{2} & . & . & . & i_{n-2} & n \\
        1 & 2 & . & . & . & n-2 & n-1
      \end{array}
    ) & B(\begin{array}{ccccccc}
            i_{1} & i_{2} & . & . & . & i_{n-2} & n \\
            j_{1} & j_{2} & . & . & . & j_{n-2} & n
          \end{array}
    )
  \end{array}\right|~~~~~(*)$
$\\=B(\begin{array}{cccccc}
      i_{1} & i_{2} & . & . & . & i_{n-2} \\
      j_{1} & j_{2} & . & . & . & j_{n-2}
    \end{array}
)|B|,$ for any $1
\leq i_{1} < i_{2} < \cdots < i_{n-2} \leq n-1, 1 \leq j_{1} < j_{2}
< \cdots < j_{n-2} \leq n-1.$ In particular, if $|B|=0,$ then $(*)$ equals  $0.$ Here we have used the notation $\\
B(\begin{array}{cccccc}
                 i_{1} & i_{2} & . & . & . & i_{p} \\
                 j_{1} & j_{2} & . & . & . & j_{p}
               \end{array}
)=\left|\begin{array}{cccccc}
    b_{i_{1}j_{1}} & b_{i_{1}j_{2}} & . & . & . & b_{i_{1}j_{p}} \\
    b_{i_{2}j_{1}} & b_{i_{2}j_{2}} & . & . & . & b_{i_{2}j_{p}} \\
    . & . & . & . & . & . \\
    . & . & . & . & . & . \\
    . & . & . & . & . & . \\
    b_{i_{p}j_{1}} & b_{i_{p}j_{2}} & . & . & . & b_{i_{p}j_{p}}
  \end{array}\right|$ for $1 \leq p \leq n.$
\end{lemma}

To prove Lemma \ref{leal1}, we need the following Lemmas.

\begin{lemma}\label{le45}
Assume $p \geq 3,~C$ is a $p \times p$  matrix,

$$C=\left(\begin{array}{ccc}
      c_{11} & \cdots & c_{1p} \\
      \cdots & \cdots & \cdots \\
      c_{p1} & \cdots & c_{pp}
    \end{array}\right),
$$

$\\$where~$c_{ij} \in \mathbb{C}$ for all $1 \leq i,j \leq p.$ Then
\begin{equation}\label{e4.10}
{c_{11}}^{p-2}|C|=|\widetilde{C}|,
\end{equation}

$\\$ where $\widetilde{C}$ is a $(p-1) \times (p-1) $ matrix given by

$$\widetilde{C}=\left(\begin{array}{ccc}
                  \left|\begin{array}{cc}
                    c_{11} & c_{12} \\
                    c_{21} & c_{22}
                  \end{array}\right|
                   & \cdots & \left|\begin{array}{cc}
                                c_{11} & c_{1p} \\
                                c_{21} & c_{2p}
                              \end{array}\right|
                    \\
                  \cdots & \cdots & \cdots \\
                  \left|\begin{array}{cc}
                    c_{11} & c_{12} \\
                    c_{p1} & c_{p2}
                  \end{array}\right|
                   & \cdots & \left|\begin{array}{cc}
                                c_{11} & c_{1p} \\
                                c_{p1} & c_{pp}
                              \end{array}\right|
                \end{array}\right).
$$
That is, $\widetilde{C}=(\widetilde{c}_{ij})_{1 \leq i \leq (p-1), 1 \leq j \leq (p-1)},$ with $\widetilde{c}_{ij}=\left|\begin{array}{cc}
                    c_{11} & c_{1(j+1)} \\
                    c_{(i+1)1} & c_{(i+1)(j+1)}
                  \end{array}\right|.$
\end{lemma}
\begin{proof}
When $c_{11}=0,$ (\ref{e4.10}) holds since both sides equal $0.$ Now assume $c_{11} \neq 0.$
By eliminating $c_{21},\cdots,c_{p1},$ we get,

$|C|=\left|\begin{array}{cccc}
       c_{11} & c_{12} & \cdots & c_{1p} \\
       0 & c_{22}-c_{12}\frac{c_{21}}{c_{11}} & \cdots & c_{2p}-c_{1p}\frac{c_{21}}{c_{11}} \\
       \cdots & \cdots & \cdots & \cdots \\
       0 & c_{p2}-c_{12}\frac{c_{p1}}{c_{11}} & \cdots & c_{pp}-c_{1p}\frac{c_{p1}}{c_{11}}
     \end{array}\right|={c_{11}}^{-(p-2)}|\widetilde{C}|.
$
\end{proof}

\bigskip

\begin{lemma}\label{le46}
If the determinant of a~ $3 \times 3$ matrix
$$\left|\begin{array}{ccc}
    a_{11} & a_{12} & a_{13} \\
    a_{21} & a_{22} & a_{23} \\
    a_{31} & a_{32} & a_{33}
  \end{array}\right|=0,
$$
where $a_{ij} \in \mathbb{C}$ for all $1 \leq i,j \leq 3.$ Then

$$\left|\begin{array}{cc}
    \left|\begin{array}{cc}
       a_{11} & a_{12} \\
       a_{21} & a_{22}
     \end{array}\right| & \left|\begin{array}{cc}
       a_{11} & a_{13} \\
       a_{21} & a_{23}
     \end{array}\right| \\
      &   \\
    \left|\begin{array}{cc}
       a_{11} & a_{12} \\
       a_{31} & a_{32}
     \end{array}\right|  & \left|\begin{array}{cc}
       a_{11} & a_{13} \\
       a_{31} & a_{33}
     \end{array}\right|
  \end{array}\right|=
\left|\begin{array}{cc}
    \left|\begin{array}{cc}
       a_{11} & a_{12} \\
       a_{21} & a_{22}
     \end{array}\right| & \left|\begin{array}{cc}
       a_{12} & a_{13} \\
       a_{22} & a_{23}
     \end{array}\right|  \\
      &   \\
    \left|\begin{array}{cc}
       a_{11} & a_{12} \\
       a_{31} & a_{32}
     \end{array}\right| & \left|\begin{array}{cc}
       a_{12} & a_{13} \\
       a_{32} & a_{33}
     \end{array}\right|
  \end{array}\right|=$$

$$~~~\left|\begin{array}{cc}
    \left|\begin{array}{cc}
       a_{11} & a_{12} \\
       a_{21} & a_{22}
     \end{array}\right| & \left|\begin{array}{cc}
       a_{11} & a_{13} \\
       a_{21} & a_{23}
     \end{array}\right| \\
      &    \\
    \left|\begin{array}{cc}
       a_{21} & a_{22} \\
       a_{31} & a_{32}
     \end{array}\right| & \left|\begin{array}{cc}
       a_{21} & a_{23} \\
       a_{31} & a_{33}
     \end{array}\right|
  \end{array}\right|=
  \left|\begin{array}{cc}
    \left|\begin{array}{cc}
       a_{11} & a_{12} \\
       a_{21} & a_{22}
     \end{array}\right| & \left|\begin{array}{cc}
       a_{12} & a_{13} \\
       a_{22} & a_{23}
     \end{array}\right| \\
      &    \\
    \left|\begin{array}{cc}
       a_{21} & a_{22} \\
       a_{31} & a_{32}
     \end{array}\right| & \left|\begin{array}{cc}
       a_{22} & a_{23} \\
       a_{32} & a_{33}
     \end{array}\right|
  \end{array}\right|=0.$$
\end{lemma}

\begin{proof}
Using Lemma $4.5$,
$$\left|\begin{array}{cc}
    \left|\begin{array}{cc}
       a_{11} & a_{12} \\
       a_{21} & a_{22}
     \end{array}\right| & \left|\begin{array}{cc}
       a_{11} & a_{13} \\
       a_{21} & a_{23}
     \end{array}\right| \\
      &   \\
    \left|\begin{array}{cc}
       a_{11} & a_{12} \\
       a_{31} & a_{32}
     \end{array}\right|  & \left|\begin{array}{cc}
       a_{11} & a_{13} \\
       a_{31} & a_{33}
     \end{array}\right|
  \end{array}\right|=a_{11}\left|\begin{array}{ccc}
    a_{11} & a_{12} & a_{13} \\
    a_{21} & a_{22} & a_{23} \\
    a_{31} & a_{32} & a_{33}
  \end{array}\right|,$$

$$\left|\begin{array}{cc}
    \left|\begin{array}{cc}
       a_{11} & a_{12} \\
       a_{21} & a_{22}
     \end{array}\right| & \left|\begin{array}{cc}
       a_{12} & a_{13} \\
       a_{22} & a_{23}
     \end{array}\right|  \\
      &   \\
    \left|\begin{array}{cc}
       a_{11} & a_{12} \\
       a_{31} & a_{32}
     \end{array}\right| & \left|\begin{array}{cc}
       a_{12} & a_{13} \\
       a_{32} & a_{33}
     \end{array}\right|
  \end{array}\right|=a_{12}\left|\begin{array}{ccc}
    a_{11} & a_{12} & a_{13} \\
    a_{21} & a_{22} & a_{23} \\
    a_{31} & a_{32} & a_{33}
  \end{array}\right|,$$

$$\left|\begin{array}{cc}
    \left|\begin{array}{cc}
       a_{11} & a_{12} \\
       a_{21} & a_{22}
     \end{array}\right| & \left|\begin{array}{cc}
       a_{11} & a_{13} \\
       a_{21} & a_{23}
     \end{array}\right| \\
      &    \\
    \left|\begin{array}{cc}
       a_{21} & a_{22} \\
       a_{31} & a_{32}
     \end{array}\right| & \left|\begin{array}{cc}
       a_{21} & a_{23} \\
       a_{31} & a_{33}
     \end{array}\right|
  \end{array}\right|=a_{21}\left|\begin{array}{ccc}
    a_{11} & a_{12} & a_{13} \\
    a_{21} & a_{22} & a_{23} \\
    a_{31} & a_{32} & a_{33}
  \end{array}\right|,$$

$$\left|\begin{array}{cc}
    \left|\begin{array}{cc}
       a_{11} & a_{12} \\
       a_{21} & a_{22}
     \end{array}\right| & \left|\begin{array}{cc}
       a_{12} & a_{13} \\
       a_{22} & a_{23}
     \end{array}\right| \\
      &    \\
    \left|\begin{array}{cc}
       a_{21} & a_{22} \\
       a_{31} & a_{32}
     \end{array}\right| & \left|\begin{array}{cc}
       a_{22} & a_{23} \\
       a_{32} & a_{33}
     \end{array}\right|
  \end{array}\right|=a_{22}\left|\begin{array}{ccc}
    a_{11} & a_{12} & a_{13} \\
    a_{21} & a_{22} & a_{23} \\
    a_{31} & a_{32} & a_{33}
  \end{array}\right|.$$
\end{proof}

\textbf{Proof of Lemma \ref{leal1} :} We proceed by induction on the dimension of $B.$ From Lemma \ref{le46},
 we know Lemma \ref{leal1} holds for $n=3.$ Now assume that it holds when the dimension of $B$ is
 less than or equal to $n-1.$ To prove it when the dimension is $n$, it is enough to show it for the case when $i_{1}=1,i_{2}=2,\cdots,i_{n-2}=n-2$ and $j_{1}=1,j_{2}=2,\cdots,j_{n-2}=n-2.$
 Namely, we show that

$$\left|\begin{array}{cc}
    \left|\begin{array}{cccccc}
      b_{11} & b_{12}  & . & . & . & b_{1n-1} \\
      b_{21} & b_{22} & . & . & . & b_{2n-1} \\
      . & . & . & . & . & . \\
      . & . & . & . & . & . \\
      . & . & . & . & . & . \\
      b_{n-11} & b_{n-1 2} & . & . & . & b_{n-1 n-1}
    \end{array}\right|
    & \left|\begin{array}{cccccc}
         b_{11} & . & . & . & b_{1 n-2} & b_{1n} \\
         b_{21} & . & . & . & b_{2 n-2} & b_{2n} \\
         . & . & . & . & . & . \\
         . & . & . & . & . & . \\
         . & . & . & . & . & . \\
         b_{n-11} & . & . & . & b_{n-1 n-2} & b_{n-1 n}
       \end{array}\right| \\

     &  \\
     &  \\

    \left|\begin{array}{cccccc}
      b_{11} & b_{12}  & . & . & . & b_{1n-1} \\
      . & . & . & . & . & . \\
      . & . & . & . & . & . \\
      . & . & . & . & . & . \\
      b_{n-2 1} & b_{n-2 2} & . & . & . & b_{n-2 n-1} \\
      b_{n1} & b_{n2} & . & . & . & b_{n n-1}
    \end{array}\right|&
    \left|\begin{array}{cccccc}
         b_{11} & . & . & . & b_{1 n-2} & b_{1n} \\
         . & . & . & . & . & . \\
         . & . & . & . & . & . \\
         . & . & . & . & . & . \\
         b_{n-2 1} & . & . & . & b_{n-2 n-2} & b_{n-2 n} \\
         b_{n1} & . & . & . & b_{nn-2} & b_{nn}
       \end{array}\right|
  \end{array}\right|
$$

$=B(\begin{array}{cccc}
      1 & 2 & \cdots & n-2 \\
      1 & 2 & \cdots & n-2
    \end{array}
)|B|,$ and the other cases are similar.

Now we view all terms here as rational functions in $b_{11},\cdots,b_{nn}.$ By Lemma \ref{le45},
\begin{equation}\label{e4.11}
|B|={b_{11}}^{-(n-2)}\left|\begin{array}{ccc}
                        B(\begin{array}{cc}
                            1 & 2 \\
                            1 & 2
                          \end{array}
                        ) & \cdots & B(\begin{array}{cc}
                                         1 & 2 \\
                                         1 & n
                                       \end{array}
                        ) \\
                        \cdots & \cdots & \cdots \\
                        B(\begin{array}{cc}
                            1 & n \\
                            1 & 2
                          \end{array}
                        ) & \cdots & B(\begin{array}{cc}
                                         1 & n \\
                                         1 & n
                                       \end{array}
                        )
                      \end{array}\right|
\end{equation}
By applying Lemma \ref{le45} and the induction hypothesis, it follows that

$$\left|\begin{array}{ccc}
                        B(\begin{array}{cc}
                            1 & 2 \\
                            1 & 2
                          \end{array}
                        ) & \cdots & B(\begin{array}{cc}
                                         1 & 2 \\
                                         1 & n
                                       \end{array}
                        ) \\
                        \cdots & \cdots & \cdots \\
                        B(\begin{array}{cc}
                            1 & n \\
                            1 & 2
                          \end{array}
                        ) & \cdots & B(\begin{array}{cc}
                                         1 & n \\
                                         1 & n
                                       \end{array}
                        )
                      \end{array}\right|=\left(B(\begin{array}{cc}
                                             1 & 2 \\
                                             1 & 2
                                           \end{array}
                      )\right)^{-(n-3)}{b_{11}}^{n-2}\left|\begin{array}{ccc}
                                             B(\begin{array}{ccc}
                                                 1 & 2 & 3 \\
                                                 1 & 2 & 3
                                               \end{array}
                                             )&\cdots  & B(\begin{array}{ccc}
                                                             1 & 2 & 3 \\
                                                             1 & 2 & n
                                                           \end{array}
                                             ) \\
                                             \cdots & \cdots & \cdots \\
                                             B(\begin{array}{ccc}
                                                 1 & 2 & n \\
                                                 1 & 2 & 3
                                               \end{array}
                                             ) & \cdots & B(\begin{array}{ccc}
                                                             1 & 2 & n \\
                                                             1 & 2 & n
                                                           \end{array})
                                           \end{array}\right|.
                      $$
Combining it with (\ref{e4.11}), we obtain
$$|B|=\left(B(\begin{array}{cc}
                                             1 & 2 \\
                                             1 & 2
                                           \end{array}
                      )\right)^{-(n-3)}\left|\begin{array}{ccc}
                                             B(\begin{array}{ccc}
                                                 1 & 2 & 3 \\
                                                 1 & 2 & 3
                                               \end{array}
                                             )&\cdots  & B(\begin{array}{ccc}
                                                             1 & 2 & 3 \\
                                                             1 & 2 & n
                                                           \end{array}
                                             ) \\
                                             \cdots & \cdots & \cdots \\
                                             B(\begin{array}{ccc}
                                                 1 & 2 & n \\
                                                 1 & 2 & 3
                                               \end{array}
                                             ) & \cdots & B(\begin{array}{ccc}
                                                             1 & 2 & n \\
                                                             1 & 2 & n
                                                           \end{array})
                                           \end{array}\right|.$$
By further applications of Lemma \ref{le45} and the induction
hypothesis as above, we arrive at the conclusion.

\medskip
Finally we state the following simple lemma:
\medskip

\begin{lemma}\label{leal2}
 Let ${\bf{b}}_{1},\cdots,{\bf{b}}_{n}$ and ${\bf{a}}$ be  $n$-dimensional column vectors with entries in $\mathbb{C}$, and let $B=({\bf{b}}_{1},\cdots,{\bf{b}}_{n})$ denote the $n\times n$ matrix.  Assume that $\mathrm{det}B \neq 0,$ and that
 $\mathrm{det}({\bf{b}}_{i_{1}},{\bf{b}}_{i_{2}},\cdots,{\bf{b}}_{i_{n-1}},{\bf{a}})=0$ for any $1 \leq i_{1} < i_{2} < \cdots < i_{n-1} \leq n.$ Then $\bf{a}=0,$ where $\bf{0}$ is the $n$-dimensional zero column vector.
\end{lemma}

\begin{proof}
Note that $\{{\bf{b}}_{1},\cdots,{\bf{b}}_{n}\}$ is a linearly independent set in $\mathbb{C}^{n}.$  Write ${\bf{a}}=\sum_{j=1}^{n}\lambda_{j} {\bf{b}}_{j}$ for some $\lambda_{j} \in \mathbb{C}, 1 \leq j \leq n.$ It is easy to see that all the $\lambda_{j}=0$ by using the assumption that $\mathrm{det}({\bf{b}}_{i_{1}},{\bf{b}}_{i_{2}},\cdots,{\bf{b}}_{i_{n-1}},{\bf{a}})=0, \,\forall 1 \leq i_{1} < i_{2} < \cdots < i_{n-1} \leq n.$
\end{proof}

Theorem \ref{thmb} will follow from:

\begin{theorem}\label{thsmth}
Let $M,M',F$ be as in Theorem 2.5 and $p \in M$ be a point with $\mathrm{rank}_l(F,p)=n+l$ for some $1\leq l\leq k-1$. Assume that in some neighborhood $O$ of $p$, $\mathrm{rank}_{l+1}(F,q)=n+l$ for all $q \in O$.  Then F is smooth near $p.$
\end{theorem}

\begin{proof}
 Assume $p=0$. Applying Lemma \ref{lenorm}, after a suitable holomorphic change of coordinates in $\mathbb{C}^{n+k},$ there exist multiindices $\{\beta_{n+1},\cdots,\beta_{n+l-1}\}$ with $1 < |\beta_{i}| \leq l~\text{for all}~n \leq i \leq n+l-1$ satisfying
\begin{equation}\label{eqbno}
{\bf{a}}|_{p}=(0,\cdots,0,\frac{\sqrt{-1}}{2}),
\left(\begin{array}{c}
  L_{1}{{\bf{a}}}|_{p} \\
  \cdots \\
  L_{n}{{\bf{a}}}|_{p} \\
  L^{\beta_{n+1}}{{\bf{a}}}|_{p}  \\
  \cdots \\
  L^{\beta_{n+l-1}}{{\bf{a}}}|_{p}
\end{array}\right)=\left(\begin{array}{ccc}
                     {\bf{B}}_{n+l-1} & {\bf{0}} & {\bf{b}}
                   \end{array}\right).
\end{equation}
Here ${\bf{B}}_{n+l-1}$ is an invertible $(n+l-1)\times (n+l-1)$
matrix, ${\bf{0}}$ is an $(n+l-1) \times(k-l)$ zero matrix,
${\bf{b}}$ is an $(n+l-1)-$dimensional column vector. From equation
(\ref{eqbno}),  we know that

\begin{equation}\label{eqalnde}
\left|\begin{array}{cccc}
  a_{1} & \cdots & a_{n+l-1} & a_{n+k} \\
   L_{1}a_{1} & \cdots & L_{1}a_{n+l-1} & L_{1}a_{n+k} \\
  \cdots & \cdots & \cdots & \cdots \\
  L_{n}a_{1} & \cdots & L_{n}a_{n+l-1} & L_{n}a_{n+k} \\
  L^{\beta_{n+1}}a_{1} & \cdots & L^{\beta_{n+1}}a_{n+l-1} & L^{\beta_{n+1}}a_{n+k}  \\
  \cdots & \cdots & \cdots & \cdots \\
  L^{\beta_{n+l-1}}a_{1} & \cdots & L^{\beta_{n+l-1}}a_{n+l-1} & L^{\beta_{n+l-1}}a_{n+k}
\end{array}\right| \neq 0~\text{at}~p.
\end{equation}
To simplify the notation, we denote the $n-$dimensional multiindices by $\beta_{0}=(0,\cdots,0),$ and  $\beta_{\mu}=(0,\cdots,0,1,0,\cdots,0),$ for $\mu=1,\cdots,n,$ where $1$ is at the $\mu^{\rm th}$ position. That is, $L^{\beta_{\mu}}=L_{\mu},\mu=1,\cdots,n.$ Then inequality (\ref{eqalnde}) can be written as
\begin{equation}\label{eqnondege}
\left|\begin{array}{cccc}
  L^{\beta_{0}}a_{1} & \cdots & L^{\beta_{0}}a_{n+l-1} & L^{\beta_{0}}a_{n+k} \\
  \cdots & \cdots & \cdots & \cdots \\
  L^{\beta_{n+l-1}}a_{1} & \cdots & L^{\beta_{n+l-1}}a_{n+l-1} & L^{\beta_{n+l-1}}a_{n+k}
\end{array}\right| \neq 0~\text{at}~p.
\end{equation}
By shrinking $O$ if necessary, it is nonzero everywhere in $O.$
Since $\mathrm{rank}_{l+1}(F,q) \leq n+l~\text{in}~O,$  we have
$$\mathrm{dim}_{\mathbb{C}}(E_{l+1}(q))=\mathrm{dim}_{\mathbb{C}}(\mathrm{Span}_{\mathbb{C}}\{(L^{\alpha}a_{1},\cdots,
L^{\alpha}a_{n+k})|_{q}:0 \leq |\alpha| \leq l+1\}) \leq n+l$$
everywhere in $O.$ Hence for any multiindex $\widetilde{\beta}$ with $0 \leq |\widetilde{\beta}| \leq l+1,$ and any $ n+l \leq j \leq n+k-1, $ we have, in $O,$
\begin{equation}\label{eqmatde}
\left|\begin{array}{ccccc}
  L^{\beta_{0}}a_{1} & \cdots & L^{\beta_{0}}a_{n+l-1} & L^{\beta_{0}}a_{n+k} & L^{\beta_{0}}a_{j} \\
  \cdots & \cdots & \cdots & \cdots & \cdots \\
  L^{\beta_{n+l-1}}a_{1} & \cdots & L^{\beta_{n+l-1}}a_{n+l-1} & L^{\beta_{n+l-1}}a_{n+k} & L^{\beta_{n+l-1}}a_{j} \\
  L^{\widetilde{\beta}}a_{1} & \cdots & L^{\widetilde{\beta}}a_{n+l-1} & L^{\widetilde{\beta}}a_{n+k} & L^{\widetilde{\beta}}a_{j}
\end{array}\right|\equiv 0.
\end{equation}
 Furthermore, we will prove the following claim.

{\bf{Claim}:} For any $1 \leq \nu \leq n, n+l \leq j \leq  n+k-1, $ and $i_{1}<i_{2}<\cdots<i_{n+l-1}$ with $\{i_{1},\cdots,i_{n+l-1}\} \subset \{1,\cdots,n+l-1,n+k\},$ the following holds in $O:$
\begin{equation}\label{eqlmatz}
L_{\nu}\left(\frac{\left|\begin{array}{cccc}
               L^{\beta_{0}}a_{i_{1}} & \cdots & L^{\beta_{0}}a_{i_{n+l-1}} & L^{\beta_{0}}a_{j}\\
               L^{\beta_{1}}a_{i_{1}} & \cdots & L^{\beta_{1}}a_{i_{n+l-1}} & L^{\beta_{1}}a_{j}\\
               \cdots & \cdots & \cdots & \cdots \\
               L^{\beta_{n+l-1}}a_{i_{1}} & \cdots & L^{\beta_{n+l-1}}a_{i_{n+l-1}} & L^{\beta_{n+l-1}}a_{j}
             \end{array}\right|
}{\left|\begin{array}{cccc}
    L^{\beta_{0}}a_{1} & \cdots & L^{\beta_{0}}a_{n+l-1} & L^{\beta_{0}}a_{n+k} \\
    \cdots & \cdots & \cdots & \cdots \\
    L^{\beta_{n+l-1}}a_{1} & \cdots & L^{\beta_{n+l-1}}a_{n+l-1} & L^{\beta_{n+l-1}}a_{n+k}
  \end{array}\right|
}\right)\equiv 0.
\end{equation}

\begin{proof}
By the quotient rule,
$$\text{the numerator of}~\left( L_{\nu}\left(\frac{\left|\begin{array}{cccc}
               L^{\beta_{0}}a_{i_{1}} & \cdots & L^{\beta_{0}}a_{i_{n+l-1}} & L^{\beta_{0}}a_{j}\\
               L^{\beta_{1}}a_{i_{1}} & \cdots & L^{\beta_{1}}a_{i_{n+l-1}} & L^{\beta_{1}}a_{j}\\
               \cdots & \cdots & \cdots & \cdots \\
               L^{\beta_{n+l-1}}a_{i_{1}} & \cdots & L^{\beta_{n+l-1}}a_{i_{n+l-1}} & L^{\beta_{n+l-1}}a_{j}
             \end{array}\right|
}{\left|\begin{array}{cccc}
    L^{\beta_{0}}a_{1} & \cdots & L^{\beta_{0}}a_{n+l-1} & L^{\beta_{0}}a_{n+k} \\
    \cdots & \cdots & \cdots & \cdots \\
    L^{\beta_{n+l-1}}a_{1} & \cdots & L^{\beta_{n+l-1}}a_{n+l-1} & L^{\beta_{n+l-1}}a_{n+k}
  \end{array}\right|}\right)\right)=$$

\begin{frame}
\footnotesize
\arraycolsep=1pt
\medmuskip = 1mu
\[
\left|\begin{array}{@{}cc@{}}
    \left|\begin{array}{cccc}
    L^{\beta_{0}}a_{1} & \cdots & L^{\beta_{0}}a_{n+l-1} & L^{\beta_{0}}a_{n+k} \\
    \cdots & \cdots & \cdots & \cdots \\
    L^{\beta_{n+l-1}}a_{1} & \cdots & L^{\beta_{n+l-1}}a_{n+l-1} & L^{\beta_{n+l-1}}a_{n+k}
  \end{array}\right| &~~~~\left|\begin{array}{cccc}
               L^{\beta_{0}}a_{i_{1}} & \cdots & L^{\beta_{0}}a_{i_{n+l-1}} & L^{\beta_{0}}a_{j}\\
               L^{\beta_{1}}a_{i_{1}} & \cdots & L^{\beta_{1}}a_{i_{n+l-1}} & L^{\beta_{1}}a_{j}\\
               \cdots & \cdots & \cdots & \cdots \\
               L^{\beta_{n+l-1}}a_{i_{1}} & \cdots & L^{\beta_{n+l-1}}a_{i_{n+l-1}} & L^{\beta_{n+l-1}}a_{j}
             \end{array}\right|
\\
    L_{\nu}\left|\begin{array}{cccc}
    L^{\beta_{0}}a_{1} & \cdots & L^{\beta_{0}}a_{n+l-1} & L^{\beta_{0}}a_{n+k} \\
    \cdots & \cdots & \cdots & \cdots \\
    L^{\beta_{n+l-1}}a_{1} & \cdots & L^{\beta_{n+l-1}}a_{n+l-1} & L^{\beta_{n+l-1}}a_{n+k}
  \end{array}\right| &~~L_{\nu}\left|\begin{array}{cccc}
               L^{\beta_{0}}a_{i_{1}} & \cdots & L^{\beta_{0}}a_{i_{n+l-1}} & L^{\beta_{0}}a_{j}\\
               L^{\beta_{1}}a_{i_{1}} & \cdots & L^{\beta_{1}}a_{i_{n+l-1}} & L^{\beta_{1}}a_{j}\\
               \cdots & \cdots & \cdots & \cdots \\
               L^{\beta_{n+l-1}}a_{i_{1}} & \cdots & L^{\beta_{n+l-1}}a_{i_{n+l-1}} & L^{\beta_{n+l-1}}a_{j}
             \end{array}\right|
  \end{array}\right|=
\]

\[
\left|\begin{array}{@{}cc@{}}
       \left|\begin{array}{cccc}
    L^{\beta_{0}}a_{1} & \cdots & L^{\beta_{0}}a_{n+l-1} & L^{\beta_{0}}a_{n+k} \\
    \cdots & \cdots & \cdots & \cdots \\
    L^{\beta_{n+l-1}}a_{1} & \cdots & L^{\beta_{n+l-1}}a_{n+l-1} & L^{\beta_{n+l-1}}a_{n+k}
  \end{array}\right|   & \left|\begin{array}{cccc}
               L^{\beta_{0}}a_{i_{1}} & \cdots & L^{\beta_{0}}a_{i_{n+l-1}} & L^{\beta_{0}}a_{j}\\
               L^{\beta_{1}}a_{i_{1}} & \cdots & L^{\beta_{1}}a_{i_{n+l-1}} & L^{\beta_{1}}a_{j}\\
               \cdots & \cdots & \cdots & \cdots \\
               L^{\beta_{n+l-1}}a_{i_{1}} & \cdots & L^{\beta_{n+l-1}}a_{i_{n+l-1}} & L^{\beta_{n+l-1}}a_{j}
             \end{array}\right| \\
           &  \\
        \left|\begin{array}{cccc}
          L_{\nu}L^{\beta_{0}}a_{1} & \cdots & L_{\nu}L^{\beta_{0}}a_{n+l-1} & L_{\nu}L^{\beta_{0}}a_{n+k} \\
          L^{\beta_{1}}a_{1} & \cdots &  L^{\beta_{1}}a_{n+l-1} &  L^{\beta_{1}}a_{n+k} \\
          \cdots & \cdots & \cdots & \cdots \\
          L^{\beta_{n+l-1}}a_{1} & \cdots & L^{\beta_{n+l-1}}a_{n+l-1} & L^{\beta_{n+l-1}}a_{n+k}
        \end{array}\right|
            & \left|\begin{array}{cccc}
               L_{\nu}L^{\beta_{0}}a_{i_{1}} & \cdots & L_{\nu}L^{\beta_{0}}a_{i_{n+l-1}} & L_{\nu}L^{\beta_{0}}a_{j}\\
               L^{\beta_{1}}a_{i_{1}} & \cdots & L^{\beta_{1}}a_{i_{n+l-1}} & L^{\beta_{1}}a_{j}\\
               \cdots & \cdots & \cdots & \cdots \\
               L^{\beta_{n+l-1}}a_{i_{1}} & \cdots & L^{\beta_{n+l-1}}a_{i_{n+l-1}} & L^{\beta_{n+l-1}}a_{j}
             \end{array}\right|
        \end{array}\right|+{\cdots}+
\]

\[
\left|\begin{array}{@{}cc@{}}
       \left|\begin{array}{cccc}
    L^{\beta_{0}}a_{1} & \cdots & L^{\beta_{0}}a_{n+l-1} & L^{\beta_{0}}a_{n+k} \\
    \cdots & \cdots & \cdots & \cdots \\
    L^{\beta_{n+l-1}}a_{1} & \cdots & L^{\beta_{n+l-1}}a_{n+l-1} & L^{\beta_{n+l-1}}a_{n+k}
  \end{array}\right|   & \left|\begin{array}{cccc}
               L^{\beta_{0}}a_{i_{1}} & \cdots & L^{\beta_{0}}a_{i_{n+l-1}} & L^{\beta_{0}}a_{j}\\
               L^{\beta_{1}}a_{i_{1}} & \cdots & L^{\beta_{1}}a_{i_{n+l-1}} & L^{\beta_{1}}a_{j}\\
               \cdots & \cdots & \cdots & \cdots \\
               L^{\beta_{n+l-1}}a_{i_{1}} & \cdots & L^{\beta_{n+l-1}}a_{i_{n+l-1}} & L^{\beta_{n+l-1}}a_{j}
             \end{array}\right| \\
           &  \\
        \left|\begin{array}{cccc}
          L^{\beta_{0}}a_{1} & \cdots & L^{\beta_{0}}a_{n+l-1} & L^{\beta_{0}}a_{n+k} \\
          \cdots & \cdots & \cdots  &  \cdots \\
          L^{\beta_{n+l-2}}a_{1} & \cdots & L^{\beta_{n+l-2}}a_{n+l-1} & L^{\beta_{n+l-2}}a_{n+k} \\
          L_{\nu}L^{\beta_{n+l-1}}a_{1} & \cdots & L_{\nu}L^{\beta_{n+l-1}}a_{n+l-1} & L_{\nu}L^{\beta_{n+l-1}}a_{n+k}
        \end{array}\right|    & \left|\begin{array}{cccc}
               L^{\beta_{0}}a_{i_{1}} & \cdots & L^{\beta_{0}}a_{i_{n+l-1}} & L^{\beta_{0}}a_{j}\\
               \cdots & \cdots & \cdots & \cdots \\
               L^{\beta_{n+l-2}}a_{i_{1}} & \cdots & L^{\beta_{n+l-2}}a_{i_{n+l-1}} & L^{\beta_{n+l-2}}a_{j} \\
               L_{\nu}L^{\beta_{n+l-1}}a_{i_{1}} & \cdots & L_{\nu}L^{\beta_{n+l-1}}a_{i_{n+l-1}} & L_{\nu}L^{\beta_{n+l-1}}a_{j}
             \end{array}\right|
        \end{array}\right|.
\]
\end{frame}

From equation (\ref{eqmatde}) and Lemma \ref{leal1}, we know each term on the right-hand side of the equation above equals $0.$ Hence equation (\ref{eqlmatz}) holds. This completes the proof of the claim.
\end{proof}

\bigskip

Thus the fraction in the parentheses in equation (\ref{eqlmatz})
equals a  $C^{k-l}$ CR
function in $O$. It follows that for any fixed $n+l \leq j \leq n+k-1,$ there
exist $C^{k-l}-$smooth CR functions
$G_{1}^{j},G_{2}^j,\cdots,G_{n+l-1}^j,G_{n+k}^j$ in $O,$ such
that, if $i_{1} < i_{2} < \cdots <i_{n+l-1}$ and
$(i_{1},i_{2},\cdots,i_{n+l-1})=(1,2,\cdots,\widehat{i_{0}},\cdots,n+l-1,n+k),i_{0}
\in\{1,2,\cdots,n+l-1,n+k\}$ (where
$(1,2,\cdots,\widehat{i_{0}},\cdots,n+l-1,n+k)$ means
$(1,2,\cdots,n+l-1,n+k)$ with the component $`` i_{0} "$ missing) then in $O,$

$$\left|\begin{array}{cccc}
               L^{\beta_{0}}a_{i_{1}} & \cdots & L^{\beta_{0}}a_{i_{n+l-1}} & L^{\beta_{0}}a_{j}\\
               L^{\beta_{1}}a_{i_{1}} & \cdots & L^{\beta_{1}}a_{i_{n+l-1}} & L^{\beta_{1}}a_{j}\\
               \cdots & \cdots & \cdots & \cdots \\
               L^{\beta_{n+l-1}}a_{i_{1}} & \cdots & L^{\beta_{n+l-1}}a_{i_{n+l-1}} & L^{\beta_{n+l-1}}a_{j}
             \end{array}\right|$$
$$=G_{i_{0}}^j\left|\begin{array}{cccc}
               L^{\beta_{0}}a_{i_{1}} & \cdots & L^{\beta_{0}}a_{i_{n+l-1}} & L^{\beta_{0}}a_{i_{0}}\\
               L^{\beta_{1}}a_{i_{1}} & \cdots & L^{\beta_{1}}a_{i_{n+l-1}} & L^{\beta_{1}}a_{i_{0}}\\
               \cdots & \cdots & \cdots & \cdots \\
               L^{\beta_{n+l-1}}a_{i_{1}} & \cdots & L^{\beta_{n+l-1}}a_{i_{n+l-1}} & L^{\beta_{n+l-1}}a_{i_{0}}
             \end{array}\right|.$$
That is,
\begin{equation} \label{eqpij}
\left|\begin{array}{cccc}
               L^{\beta_{0}}a_{i_{1}} & \cdots & L^{\beta_{0}}a_{i_{n+l-1}} & L^{\beta_{0}}(a_{j}-G_{i_{0}}^{j}a_{i_{0}})\\
               L^{\beta_{1}}a_{i_{1}} & \cdots & L^{\beta_{1}}a_{i_{n+l-1}} & L^{\beta_{1}}(a_{j}-G_{i_{0}}^{j}a_{i_{0}})\\
               \cdots & \cdots & \cdots & \cdots \\
               L^{\beta_{n+l-1}}a_{i_{1}} & \cdots & L^{\beta_{n+l-1}}a_{i_{n+l-1}} & L^{\beta_{n+l-1}}(a_{j}-G_{i_{0}}^{j}a_{i_{0}})
             \end{array}\right|\equiv 0.
\end{equation}
We further assert:

\bigskip
{\bf{Claim:}} In $O,$ we have,

\begin{equation} \label{eqzaj}
\left|\begin{array}{cccc}
               L^{\beta_{0}}a_{s_{1}} & \cdots & L^{\beta_{0}}a_{s_{n+l-1}} & L^{\beta_{0}}(a_{j}-\sum_{i=1}^{n+l-1}G_{i}^{j}a_{i}-G_{n+k}^{j}a_{n+k})\\
               L^{\beta_{1}}a_{s_{1}} & \cdots & L^{\beta_{1}}a_{s_{n+l-1}} & L^{\beta_{1}}(a_{j}-\sum_{i=1}^{n+l-1}G_{i}^{j}a_{i}-G_{n+k}^{j}a_{n+k})\\
               \cdots & \cdots & \cdots & \cdots \\
               L^{\beta_{n+l-1}}a_{s_{1}} & \cdots & L^{\beta_{n+l-1}}a_{s_{n+l-1}} & L^{\beta_{n+l-1}}(a_{j}-\sum_{i=1}^{n+l-1}G_{i}^{j}a_{i}-G_{n+k}^{j}a_{n+k})
             \end{array}\right| \equiv 0
\end{equation}
for all $s_{1} < s_{2} <\cdots <s_{n+l-1}$ with $\{s_{1},\cdots,s_{n+l-1}\} \subset \{1,\cdots,n+l-1,n+k\}$ and any $n+l \leq j \leq n+k-1.$

\begin{proof}
Assume that $(s_{1},\cdots,s_{n+l-1})=(1,\cdots,\widehat{s_{0}},\cdots,n+l-1,n+k).$ Notice that for any $n+l \leq j \leq n+k-1, i \neq s_{0}$ and $i \in \{1,\cdots,n+l-2,n+k\},$
\begin{equation}
\left|\begin{array}{cccc}
               L^{\beta_{0}}a_{s_{1}} & \cdots & L^{\beta_{0}}a_{s_{n+l-1}} & L^{\beta_{0}}(G_{i}^{j}a_{i})\\
               L^{\beta_{1}}a_{s_{1}} & \cdots & L^{\beta_{1}}a_{s_{n+l-1}} & L^{\beta_{1}}(G_{i}^{j}a_{i})\\
               \cdots & \cdots & \cdots & \cdots \\
               L^{\beta_{n+l-1}}a_{s_{1}} & \cdots & L^{\beta_{n+l-1}}a_{s_{n+l-1}} & L^{\beta_{n+l-1}}(G_{i}^{j}a_{i})
             \end{array}\right| \equiv 0.
\end{equation}
Combining this with equation (\ref{eqpij}), one can check that equation (\ref{eqzaj}) holds.
\end{proof}

By Lemma \ref{leal2}, equation (\ref{eqnondege}), and (\ref{eqzaj}), we immediately obtain that in $O,$
$$L^{\beta_{t}}(a_{j}-\sum_{i=1}^{n+l-1}G_{i}^{j}a_{i}-G_{n+k}^{j}a_{n+k})=0, \forall~1 \leq t \leq n+l-1,n+l \leq j \leq n+k-1.$$
In particular, when $t=0,$ we have:
\begin{equation}
a_{j}-\sum_{i=1}^{n+l-1}G_{i}^{j}a_{i}-G_{n+k}^{j}a_{n+k}=0,  n+l \leq j \leq n+k-1.
\end{equation}
That is, in $O,$
\begin{equation}\label{eqneG}
F_{j}+\overline{\phi_{z'_{j}}^{*}}-\sum_{i=1}^{n+l-1}\overline{G_{i}^{j}}(F_{i}+\overline{\phi_{z'_{i}}^{*}})-
\overline{G_{n+k}^{j}}(\frac{1}{2\sqrt{-1}}+\overline{\phi_{z'_{n+k}}^{*}})=0.
\end{equation}
Recall that we have, by shrinking $O$ if necessary, in $O,$

\begin{equation}\label{eqFde}
-\frac{F_{n+k}-\overline{F_{n+k}}}{2\sqrt{-1}}+F_{1}\overline{F_{1}}+
\cdots+F_{n+k-1}\overline{F_{n+k-1}}+\phi^{*}(F,\overline{F})=0,
\end{equation}
\begin{equation}\label{eqlFg}
\frac{L_{j}\overline{F_{n+k}}}{2\sqrt{-1}}+F_{1}L_{j}\overline{F_{1}}+\cdots+F_{n+k-1}
L_{j}\overline{F_{n+k-1}}+L_{j}\phi^{*}(F,\overline{F})=0, 1 \leq j \leq n,
\end{equation}
\begin{equation}\label{eqlFd}
\frac{L^{\beta_{t}}\overline{F_{n+k}}}{2\sqrt{-1}}+F_{1}L^{\beta_{t}}\overline{F_{1}}+\cdots+F_{n+k-1}
L^{\beta_{t}}\overline{F_{n+k-1}}+L^{\beta_{t}}\phi^{*}(F,\overline{F})=0, n+1 \leq t \leq n+l-1.
\end{equation}

We introduce local coordinates $(x,y,s)\in \mathbb R^n\times \mathbb R^n\times \mathbb R^d$ that vanish at the central ponit $p\in M$. By Theorem 2.9,  $G_{i}^{j},G_{n+k}^{j},F_{1},\cdots,F_{n+k}$ extend to almost analytic functions into a wedge $\{(x,y,s+it)\in U\times V\times \Gamma_1: \,(x,y,s)\in U\times V,\,t\in \Gamma_1\}$, with edge $M$ near $p=0$ for all $1 \leq i \leq n+l-1, n+l \leq j \leq n+k-1.$ Here $U\times V$ is a neighborhood of the origin in $\mathbb{C}^n \times \mathbb R^d$ and $\Gamma_1$ is an acute convex cone in $\mathbb R^d$ in $t-$space. We still denote the extended functions by $G_{i}^{j},G_{n+k}^{j},F_{1},\cdots,F_{n+k}.$  Arguments similar to those used in the proof of Theorem 2.3 imply that the $G_{i}^{j}$ and $G_{n+k}^{j}$ satisfy the estimates:
\[
\left |D_x^{\alpha}D_y^{\beta}D_s^{\gamma}G_{i}^{j}(z,s,t)\right |\leq \frac{C}{|t|^{\lambda}}, \left |D_x^{\alpha}D_y^{\beta}D_s^{\gamma}G_{n+k}^{j}(z,s,t)\right |\leq \frac{C}{|t|^{\lambda}},\,\,\text{for some $C,\lambda>0$}
\]
and
\[
 D_x^{\alpha}D_y^{\beta}D_s^{\gamma}\overline{\partial}_{w_\nu}G_{i}^{j}(z,s,t)=O(|t|^m),
 D_x^{\alpha}D_y^{\beta}D_s^{\gamma}\overline{\partial}_{w_\nu}G_{n+k}^{j}(z,s,t)=O(|t|^m),\,\,
 \]
for all $1 \leq i \leq n+l-1,n+l \leq j \leq n+k-1,1 \leq \nu \leq d,m \geq 1.$  Similar estimates hold for $F_{1},\cdots,F_{n+k}.$

\bigskip

We now use equations $(4.22), (4.23), (4.24)$ and $(4.21)$ to get a smooth map $\Psi(Z',\overline{Z'},W)=(\Psi_{1},\cdots,\Psi_{n+k})$ defined in a neighborhood of $\{0\} \times \mathbb{C}^{q}$ in $\mathbb{C}^{n+k} \times \mathbb{C}^{q},$~ smooth in the first $n+k$ variables and polynomial in the last $q$ variables for some integer $q$, such that,
$$\Psi(F,\overline{F},(L^{\alpha}\overline{F})_{1 \leq |\alpha| \leq l},
\overline{G_{1}^{n+l}},\cdots,\overline{G^{n+l}_{n+l-1}},\overline{G_{n+k}^{n+l}},\cdots,\overline{G_{1}^{n+k-1}},\cdots,
\overline{G_{n+l-1}^{n+k-1}},\overline{G_{n+k}^{n+k-1}})=0$$
at $(z,s,0)$ with $(z,s) \in U \times V.$ Write $$\overline{G}=(\overline{G_{1}^{n+l}},\cdots,\overline{G^{n+l}_{n+l-1}},\overline{G_{n+k}^{n+l}},\cdots,\overline{G_{1}^{n+k-1}},\cdots,
\overline{G_{n+l-1}^{n+k-1}},\overline{G_{n+k}^{n+k-1}}).$$
Observe that
$$\Psi_{Z'}|_{(p,(L^{\alpha}\overline{F})_{1\leq |\alpha|\leq l}(p),\overline{G}(p))}=\left(\begin{array}{ccc}
                   {\bf{0}}_{n+l-1} & {\bf{0}}_{k-l} & \frac{\sqrt{-1}}{2} \\
                   {\bf{B}}_{n+l-1} & {\bf{0}} & {\bf{b}} \\
                   {\bf{C}} & {\bf{I}}_{k-l} & {\bf{0}}^{t}_{k-l}
                 \end{array}\right),
$$
where ${\bf{0}}_{N}$ is an $N-$dimensional zero row vector,
${\bf{C}}$ is a $(k-l)\times (n+l-1)$ matrix, ${\bf{I}}_{k-l}$
is the  $(k-l) \times (k-l)$ identity matrix and we recall that ${\bf{B}}_{n+l-1}$ is an invertible $(n+l-1)\times
(n+l-1)$ matrix, ${\bf{0}}$ is an $(n+l-1) \times(k-l)$ zero
matrix, ${\bf{b}}$ is an $(n+l-1)-$dimensional column vector.

 The matrix $\Psi_{Z'}|_{(p,(L^{\alpha}\overline{F})_{1 \leq |\alpha|\leq l}(p),\overline{G}(p))}$ is invertible. By applying Theorem $3.1$, we get a solution $\psi=(\psi_{1},\cdots,\psi_{n+k})$ satisfying (\ref{e2.4}) and for each $1 \leq j \leq n+k,$

$$F_{j}=\psi_{j}(F,\overline{F},(L^{\alpha}\overline{F})_{1 \leq |\alpha| \leq l},
\overline{G})$$
at $(z,s,0)$ with $(z,s) \in U \times V.$
Recall that in the proof of Theorem 2.3, for each $i=1,\cdots,n,$ we denoted by $M_{i}$ the smooth extension of $L_{i}$ to $U \times V \times \mathbb{R}^d$ satisfying (\ref{eqn220}).
For each $1 \leq j \leq n+k,$ set
$$h_{j}(z,s,t)=\psi_{j}(F(z,s,-t),\overline{F}(z,s,-t),(M^{\alpha}\overline{F})_{1 \leq |\alpha| \leq l}(z,s,-t),\overline{G}(z,s,-t))$$ and shrink $U$ and $V$ and
choose $\delta$ in such a way that $h_{j}$  is well defined and
continuous in $\overline{\Omega_{-}}$ where $\Omega_{-}=\{(x,y,s+it):\,(x,y,s)\in U\times V,\,t\in -\Gamma_1, |t| \leq \delta \}$. The same proof as before leads to the estimates:

\[
\left |D_x^{\alpha}D_y^{\beta}D_s^{\gamma}h_j(z,s,t)\right |\leq \frac{C}{|t|^{\lambda}},\,\,\text{for some $C,\lambda>0$}
\]
and
\[
 D_x^{\alpha}D_y^{\beta}D_s^{\gamma}\overline{\partial}_{w_\nu}h_j(z,s,t)=O(|t|^m),\,\, \forall \nu=1,\cdots,d,\, m=1,2,\dots
 \] for $t\in -\Gamma_1$, $1 \leq j \leq n+k.$

Notice that the $F_{j}$ satisfy similar estimates in $\Gamma_1$, and  $b_{+}F_{j}=b_{-}h_{j}$ for each $1\leq j \leq n+k.$ Applying Theorem V.3.7 in [BCH] as before, we conclude that $F$ is smooth near $p$. This establishes Theorem \ref{thsmth}.

\end{proof}

\textbf{Proof of Theorem \ref{thmb}:} Fix $p \in \Omega_{2}.$ Let a neighborhood $\widetilde{O}$ of $p$ and $\{p_{i}\}_{0}^{\infty} \subset \widetilde{O}$ be as  mentioned in Remark 2.7, and write $d=\mathrm{deg}(F,p).$ Since $\mathrm{rank}_{d}(F,q) \leq n+d-1$ for all $q \in \widetilde{O},$ and $\mathrm{rank}_{d-1}(F,p_{i})=n+d-1$ for all $i \geq 0,$ by Theorem $\ref{thsmth},$ $F$ is smooth near $p_{i}$ for all $i \geq 0.$ This establishes Theorem 2.8.

Theorem $2.5$ follows from Theorem $2.8$ and Theorem $2.9$.

As a consequence of Theorem \ref{thsmth}, we immediately have
\begin{Corollary}
Let $M\subset \mathbb{C}^{n+1}$, $M'\subset \mathbb{C}^{n+k}$ be two smooth
strongly pseudoconvex real hypersurfaces $(n \geq 1,k \geq 1)$, $F:~ M\rightarrow M'$ be a $C^{2}-$smooth  CR  map. Assume that $\mathrm{rank}_{2}(F,p) \leq n+1$ everywhere in $M.$ Then F is smooth.
\end{Corollary}

\begin{proof}
We may assume that $F$ is nonconstant. By a well known argument using Hopf's lemma as in the Appendix, $dF:T_{p}^{(1,0)}M \rightarrow T_{F(p)}^{(1,0)} M'$ is injective at every $p \in M.$ Note that $\mathrm{rank}_{1}(F,p)=n+1$ for all $p \in M$ by Lemma \ref{lerankc}. By Theorem \ref{thsmth} (note that in this case, the proof showed that we did not need $F$ to be $C^k$), we arrive at the conclusion.
\end{proof}

Since a CR diffeomorphism of class $C^k$ of a $k-$nondegenerate manifold is $k-$nondegenerate, Thoerem 2.3 implies the following:
\begin{Corollary}
Let $M\subset \Bbb C^N$ be a generic CR manifold that is $k_0-$nondegenerate.  Suppose $H=(H_{1},\cdots,H_{N}): M \rightarrow M$  is a CR diffeomorphism of class $C^{k_{0}}$  such that for some $p_0\in M$ and an open convex cone $\Gamma \subset \mathbb{R}^{d},$
$$\mathrm{WF}(H_{j})|_{p_{0}} \subset \Gamma, j=1,\cdots,N$$
where $d$ is the CR codimension of $M.$ Then $H$ is $C^{\infty}$ in some neighborhood of $p_{0}.$
\end{Corollary}

\section {Appendix}
\subsection{On CR mappings into a lower dimensional target}
In this appendix we include a result which shows why we don't consider the case when the target manifold has a lower CR dimension. The result is known to experts but we have presented it here since we are not aware of a  reference.
\begin{theorem}
Let  $M \subset \mathbb{C}^{N},M' \subset \mathbb{C}^{N'}$ be smooth strongly pseudoconvex real hypersurfaces with $N \geq 2,N' \geq 2.$ Let $F: M \rightarrow M'$ be a CR mapping of class $C^{2}.$ Assume that $N' < N.$ Then $F$ is a constant map.
\end{theorem}
\begin{proof}
 Suppose that $F$ is nonconstant. Fix $p \in M, p' \in M'$ with $p'=F(p).$ Choose suitable coordinates in $\mathbb{C}^{N}$ and $\mathbb{C}^{N'}$ near $p, p'$ such that: $p=0,p'=0$ and $M$ is locally defined by
$$r(Z,\overline{Z})=-\frac{z_{N}-\overline{z_{N}}}{2\sqrt{-1}}+z_{1}\overline{z}_{1}+\cdots+ z_{N-1}\overline{z_{N-1}} +\phi(Z,\overline{Z})~\text{near}~p,$$
$M'$ is locally defined by
$$\rho(Z',\overline{Z'})=-\frac{z'_{N'}-\overline{z'_{N'}}}{2\sqrt{-1}}+z'_{1}\overline{z'}_{1}+\cdots+ z'_{N'-1}\overline{z'_{N'-1}} +\phi^{*}(Z',\overline{Z'})~\text{near}~p'.$$
Here $Z=(z_{1},\cdots,z_{N}),Z'=(z'_{1},\cdots,z'_{N'})$ are the coordinates of $\mathbb{C}^{N}$ and $\mathbb{C}^{N'}$ respectively. Moreover, $\phi(Z,\overline{Z})=O(|Z|^{3}),\phi^{*}(Z',\overline{Z'})=O(|Z'|^{3})$ are real-valued smooth functions near $p,p'$ respectively. We write
$$L_{i}=\frac{\partial r}{\partial z_{N}}\frac{\partial}{\partial z_{i}}-\frac{\partial r}{\partial z_{i}}\frac{\partial}{\partial z_{N}},~~T=\sqrt{-1}(\frac{\partial r}{\partial \overline{z}_{N}}\frac{\partial}{\partial z_{N}}-\frac{\partial r}{\partial z_{N}}\frac{\partial}{\partial \overline{z}_{N}}),1 \leq i \leq N-1,$$
which are vector fields tangent to $M$ near the central point $0.$

Write $F=(F_{1},\cdots,F_{N'}).$ Near $0$ on $M$ we have
\begin{equation}\label{e1}
\rho(F,\overline{F})=-\frac{F_{N'}-\overline{F_{N'}}}{2\sqrt{-1}}+F_{1}\overline{F}_{1}+\cdots+ F_{N'-1}\overline{F_{N'-1}} +\phi^{*}(F,\overline{F})=0.
\end{equation}
Applying $L_{i},1 \leq i \leq N-1$ to equation (\ref{e1}) and evaluating at $0$, one easily gets
$$\frac{\partial F_{N'}}{\partial z_{i}}|_{0}=0,1 \leq i \leq N-1.$$
Similarly, by applying $L_{k}L_{l},1 \leq k,l \leq N-1$ to equation (\ref{e1}) and evaluating at $0,$ we get,
$$\frac{\partial^{2} F_{N'}}{\partial z_{k}\partial z_{l}}|_{0}=0,1 \leq k,l \leq N-1.$$

By the Lewy extension theorem, $F$ extends holomorphically to the pseudoconvex side of $M$ denoted by $\Omega.$ We  may assume that $\Omega$ is a union of analytic discs attached to $M$. That is, for each $q\in \Omega$, there exists a continuous function
\[
G: \overline{\Delta}\rightarrow \Bbb C^N \quad (\Delta=\{\zeta\in \Bbb C:|\zeta|<1\})
\]
analytic on $\Delta$ such that $G(\Delta)\subset \Omega$, $G(\partial \Delta)\subset M$, and $G(0)=q$. For each such $G$, the function $\rho\circ F\circ G$ is continuous on $\overline{\Delta}$, subharmonic on $\Delta$ and vanishes on $\partial \Delta$. If this function is constant for every analytic disc $G$ attached to $M$, then $F$ would map $\Omega$ into $M'$ which would contradict the strict pseudoconvexity of $M'$ unless $F$ is constant.
 This allows us to apply the maximum principle and the Hopf lemma to the subharmonic function $\rho(F,\overline{F})\leq 0$ near $p$ over $\Omega$ to conclude that
\begin{equation}
\frac{\partial}{\partial \mathrm{Im} (z_{N})}(\rho(F,\overline{F}))|_{0}= \frac{\partial}{\partial \mathrm{Im} (z_{N})}(-\mathrm{Im}(F_{N'})+ \sum_{j=1}^{N'-1}|F_{j}|^{2})|_{0}=-\lambda <0,
\end{equation}
 for some $\lambda >0,$ that is,
\begin{equation}
\sqrt{-1}(\frac{\partial}{\partial z_{N}}-\frac{\partial}{\partial \overline{z_{N}}})(-\frac{F_{N'}-\overline{F}_{N'}}{2\sqrt{-1}})|_{0}=-\frac{1}{2}(\frac{\partial F_{N'}}{\partial z_{N}}+\frac{\partial \overline{F}_{N'}}{\partial \overline{z}_{N}})|_{0}=-\frac{\partial F_{N'}}{\partial z_{N}}|_{0}=-\lambda <0.
\end{equation}
Here we have used the fact that
\begin{equation}\label{e2}
T(\rho(F,\overline{F}))=0\,\,\
\end{equation}
which implies $\frac{\partial F_{N'}}{\partial z_{N}}|_{0}=\frac{\partial \overline{F}_{N'}}{\partial \overline{z}_{N}}|_{0}.$
Hence we can write:
\begin{equation}\label{e3}
F_{N'}=\lambda z_{N}+ O(|z_{N}||\tilde{z}|+|z_{N}|^2)+o(|Z|^{2}),\,\,\,\tilde z=(z_1,\dots,z_{N-1})
\end{equation}
\begin{equation}\label{e4}
F_{j}=b_{j} z_{N}+ \sum_{i=1}^{N-1} a_{ij}z_{i} + O(|Z|^{2}), 1 \leq j \leq N'-1,
\end{equation}
where $b_{j} \in \mathbb{C}, a_{ij} \in \mathbb{C}, 1 \leq i \leq N-1, 1 \leq j \leq N'-1.$ That is,
\begin{equation}\label{e40}
(F_{1},\cdots,F_{N'-1})=z_{N}(b_{1},\cdots,b_{N'-1})+(z_{1},\cdots,z_{N-1})A+ (\hat{F}_{1},\cdots,\hat{F}_{N'-1}),
\end{equation}
where $A=(a_{ij})_{(N-1)\times(N'-1)}$ is an $(N-1) \times (N'-1)$ matrix, and $\hat{F}_{j}=O(|Z|^{2}),$ for any $1 \leq j \leq N'-1.$
Next we write $Z=(\widetilde{z},z_{N}),$ where $\widetilde{z}=(z_{1},\cdots,z_{N-1}),$  and we introduce the notion of weighted degree: For a function $h$ on $M,$ we write $h \in o_{wt}(s)$ if
$$\lim_{t \rightarrow 0^{+}} \frac{h(t\widetilde{z},t^{2}z_{N},t\overline{\widetilde{z}},t^{2}\overline{z}_{N})}{t^{s}}\rightarrow  0$$
uniformly with respect to $(\widetilde{z},z_{N})\approx 0$ in $\mathbb{C}^{N-1} \times \mathbb{C}.$ That is, we equip $\widetilde{z},z_{N}$ with weighted degrees $1,2$ respectively.

From equation (\ref{e1})
\begin{equation}\label{e5}
\frac{F_{N'}-\overline{F}_{N'}}{2 \sqrt{-1}}=(F_{1},\cdots,F_{N'-1})(\overline{F}_{1},\cdots,\overline{F}_{N'-1})^{t}+\phi^{*}(F,\overline{F}),
\end{equation}
whenever $z_{N}=u +\sqrt{-1}(|\widetilde{z}|^2 +\phi(Z,\overline{Z}))$ near $0.$  We can rewrite  equation (\ref{e5}) in terms of $u$ and $\widetilde{z}$ by using equations (\ref{e3}),(\ref{e4}),(\ref{e40}):
\begin{equation}
\lambda |\widetilde{z}|^{2}+o_{wt}(2)=(z_{1},\cdots,z_{N-1})A A^{*}(\overline{z}_{1},\cdots,\overline{z}_{N-1})^{t}+o_{wt}(2)
\end{equation}
 Then by collecting terms on both sides of weighted degree two, one easily gets,
 $$\lambda |\widetilde{z}|^{2}=(z_{1},\cdots,z_{N-1})A A^{*}(\overline{z}_{1},\cdots,\overline{z}_{N-1})^{t},$$
 which implies that,
 \begin{equation}\label{e6}
 \lambda {\bf{I}}_{N-1}=A A^{*},
 \end{equation}
 where ${\bf{I}}_{N-1}$ is the $(N-1) \times (N-1)$ identity matrix. But $A$ is an $(N-1) \times (N'-1)$ matrix with rank at most $N'-1$ and so (\ref{e6}) can not hold since $N'-1 <N-1.$

\end{proof}

\bibliographystyle{amsalpha}

\end{document}